\documentclass[10pt]{article}

\usepackage[left=3.18cm,right=3.18cm,top=2.54cm,bottom=2.54cm]{geometry}
\usepackage{amsfonts}
\usepackage{amsmath}
\numberwithin{equation}{section}
\usepackage{amssymb}
\usepackage{amsthm}
\usepackage{pgfplots}
\pgfplotsset{compat=1.18} 
\usepackage{tikz}
\usetikzlibrary{arrows}
\usepackage{bm}
\usepackage[nottoc]{tocbibind}
\usepackage{graphicx}
\usepackage{appendix}
\usepackage{titletoc}
\usepackage{tikz-cd}
\usepackage{mathtools}
\usepackage{enumerate}
\usepackage{enumitem}
\usepackage{tcolorbox}
\tcbuselibrary{skins, breakable, theorems}
\usepackage[breaklinks,colorlinks,linkcolor=blue,citecolor=red,urlcolor=red]{hyperref}
\usepackage{bookmark}

\usepackage[makeroom]{cancel}

\newtheoremstyle{theo}
{.8em}
{.8em}
{\sffamily}
{}
{\bfseries}
{}
{.5em}
{}
\theoremstyle{theo}
\newtheorem{thm}{Theorem}[section]

\newtheorem{lem}[thm]{Lemma}
\newtheorem{prop}[thm]{Proposition}
\newtheorem{cor}[thm]{Corollary}

\newtheorem{defn}[thm]{Definition}

\theoremstyle{definition}

\newtheoremstyle{rmkit}
{.8em}
{.8em}
{\itshape}
{}
{\scshape}
{}
{.5em}
{}
\theoremstyle{rmkit}

\newcommand{\secff}[0]{\overrightarrow{\textit{\uppercase\expandafter{\romannumeral2}}}}

\newcommand{\Ric}[0]{\textup{Rc}}
\newcommand{\Rm}[0]{\textup{Rm}}

\newcommand{\cL}[0]{\mathcal L}

\newcommand{\cR}[0]{\mathcal R}

\newcommand{\cQ}[0]{\mathcal Q}

\newcommand{\cS}[0]{\mathcal S}

\newcommand{\dd}[0]{\textup{d}}

\begin{document}
	\bibliographystyle{alpha}
	
	\title{\textbf{The Ricci-DeTurck flow on complete manifolds} \\ \large---In memory of Richard S. Hamilton}
	
	\author{
        {Jing-Bin Cai}
    \and
        {Bing Wang}
	}
    \date{}
	\maketitle
    
	\begin{abstract}
       Based on the framework of Koch-Lamm and tensor heat kernel estimates, we obtain a uniform proof of the short-time existence, uniqueness, and continuous  dependence for Ricci flows starting from a complete Riemannian metric with bounded curvature. 
       A new ingredient is an effective continuous dependence estimate without the assumption of injectivity radius lower bound. 
	\end{abstract}
	\tableofcontents


    \section{Introduction}

      The Ricci flow is a family of Riemannian metrics $g_t$ that satisfy the evolution equation
      \begin{align*}
          \partial_t g=-2Rc. 
      \end{align*}
      As invented by Richard S. Hamilton in 1982~\cite{hamilton1982}, it is currently a central topic in geometric analysis. The last four decades witnessed great success achieved by the Ricci flow. 

      The short-time existence of the Ricci flow on a closed manifold was firstly established by Richard S. Hamilton,  using Nash-Moser's implicit function theorem~\cite{hamilton1982implicit}.
      Later, Deturck invented a technique that makes the short-time existence much simpler~\cite{deturck1983deforming}.   
      Instead of considering equation $\partial_t g= -2Rc$, he initiated the study of
      \begin{align}
          \partial_t g =P_{\bar{g}}g= -2Rc + \mathcal{L}_{X} g, 
          \label{RDequ}
      \end{align}
      where $P_{\bar{g}}g$ is the Ricci-Deturck operator and $X=X_{\bar{g}}g$ is a smooth vector field (cf. Section~\ref{sec:compute_RD_operator} for more details).  
      In order to determine $X$, we need to fix a background metric $\bar{g}$.  Then $X^k=g^{ij}(\bar{\Gamma}_{ij}^k-\Gamma_{ij}^k)$. The direct calculation shows that equation (\ref{RDequ}) is strictly parabolic.
      Therefore, the short-time existence of the Ricci-Deturck flow on closed manifolds follows from the standard parabolic equation theory.   Then the Ricci flow can be obtained from the Ricci-Deturck flow by pulling back through the diffeomorphisms generated by $X(t)$. 
    
      If the underlying manifold $M$ is non-compact and $g_0$ has bounded curvature, then the short time existence was first obtained by Shi (cf. Theorem 4.3 of~\cite{Shi1989DeformingTM}).   His proof proceeds by solving the Dirichlet problem on each domain $D_k\times [0,T]$, where $\bigcup_{k=1}^{\infty} \{D_k\}$ is an exhaustion of $M$. He showed that the derivatives of each solution $g_{ij}(k,x,t)$ are uniformly bounded on each compact subset of $M$. Then a global solution is obtained by taking the $C^{\infty}$-limit of these local solutions. 
      Thus, the existence problem was solved.  
      The uniqueness problem was settled by Chen-Zhu~\cite{chen2006uniqueness}. 
 
      A natural question is: if we perturb the initial metrics slightly, shall we have a uniform existence time? Shall the metrics $g(t)$ depend on $g_0$ continuously?  
      Such problems were first studied by M. Simon. 
      In~\cite{simon2002deformation}, he obtained the following results. 
      On a complete manifold $M$, 
      suppose $g_0$ is a $C^0$-metric and $\bar{g}$ is a smooth metric satisfying the curvature bound $|Rm|_{\bar{g}} \leq \Lambda_0$.
      Suppose that
      \begin{align*}
          \frac{1}{1 + \epsilon} \bar g \leq g_0 \leq (1 + \epsilon)\bar g
      \end{align*}
      for some fixed small $\epsilon=\epsilon(n)$. 
      Then the Ricci-Deturck flow (using $\bar{g}$ as the background metric) initiated from $g_0$ exists for a uniform time $T(n,\Lambda_0)$ and satisfies 
     \begin{align*}
         \frac{1}{1 + 2\epsilon} \bar g \leq g_t \leq (1 + 2\epsilon)\bar g, \quad
         \forall \; t \in [0, T]. 
     \end{align*}
     
     Simon's results hold on an arbitrary complete manifold. If we specify the underlying manifold as $\mathbb{R}^n$, Koch-Lamm~\cite{koch2012geometric} introduced a powerful technique, with which they obtained more delicate estimates.
     Let $\bar g=g_E$. Assume $g_t$ is a Ricci-Deturck flow solution with background metric $\bar{g}=g_E$.
     Let $h_t := g_t - g_E$.  Then $h_t$ satisfies the perturbation equation
     \begin{align*}
         \partial_t h - \Delta h = \text{non-linear terms}.
     \end{align*}
    They constructed Banach spaces to solve the integral equation of $h$ associated with the differential equation above. If $h_0$ is sufficiently small, they proved in~\cite{koch2012geometric} the well-posedness of the perturbation equation, through the application of the contraction mapping principle.   In particular,  their results implies that
    \begin{align}
        \|h_t\|_{L^{\infty}(M)} \leq C \|h_0\|_{L^{\infty}(M)}, \quad \forall \; t>0.     \label{eqn:SJ10_1}
    \end{align}

     In the context of~\cite{koch2012geometric}, the background metric is the Euclidean metric $g_E$, which is a static point of the Ricci flow equation.  Up to normalization, it can be replaced by an Einstein metric. This is realized by the work of Bamler.
     In~\cite{bamler2014stability}, using the Einstein metric as a background metric,  Bamler studied the perturbation of the Ricci-DeTurck flow near a hyperbolic Einstein metric $\bar{g}$ with finite volume.  Let $g_t$ be the Ricci-DeTurck flow with the initial metric $g_0$ and the background $\bar{g}$.
     Also, define $h_t := g_t-\bar g$.  Among other things, he proved that $h_t$ also satisfies (\ref{eqn:SJ10_1}), although the hyperbolic Einstein metric may collapse at infinity.  
     
     In \cite{burkhardt2019pointwise}, Burkhardt-Guim considered the more general case by allowing $\bar{g}=\bar{g}_t$ to be a fixed Ricci flow.  When $M$ is closed and $h_t := g_t - \bar g_t$, she also obtained estimates such as (\ref{eqn:SJ10_1}).  The compact assumption for $M$ can be further replaced by
     the assumption of bounded geometry; see~\cite{bahuaud2024convergence}.
     There are many works concerning the smoothing of metrics and related stabilities via Ricci flow. For more information in this direction, see~\cite{chending},~\cite{schnurer2007stability},~\cite{kroncke2019stability},~\cite{deruelle2021stability},~\cite{lamm2021ricci},~\cite{chu2022ricci} and the references therein. \\

     In this paper, we adapt the framework of Koch-Lamm to study the short-time existence, uniqueness, and continuous dependence of the Ricci-Deturck flow on a complete manifold. 
     Our main result is the following theorem. 

    \begin{thm}[\textbf{Main theorem}]
    \label{thm:main}
    Let $(M^n, g_0)$ be a complete Riemannian manifold with bounded curvature $|Rm|_{g_0} \leq \Lambda_0$. 
    Then there exist constants $T=T(n,\Lambda_0)$, $\delta=\delta(n,\Lambda_0)$ and $C=C(n,\Lambda_0)$ with the following property. 

    \begin{itemize}
        \item There exists a unique Ricci flow solution $\{g_t\}_{0 \leq t \leq T}$ initiated from $g_0$. 
        \item For each continuous metric $\hat{g}_0$ satisfying 
        $\displaystyle \sup_M |\hat{g}_0-g_0|_{g_0}<\delta$, there exists a unique Ricci-Deturck flow solution $\{\hat{g}_t\}_{0 \leq t \leq T}$,  
        with background metrics $g_t$ and initial data   $\hat{g}_0$, satisfying
    \begin{align}
       \sup_{M \times [0, T]} |\hat{g}_t-g_t|_{g_t} \leq C \sup_M |\hat{g}_0-g_0|_{g_0}.
    \label{eqn:SJ08_1}    
    \end{align}
    \end{itemize}
    \end{thm}

    As mentioned, the existence and uniqueness of the Ricci flow was already known due to the work of Shi~\cite{Shi1989DeformingTM} and Chen-Zhu~\cite{chen2006uniqueness}. 
    On a closed manifold or a complete Riemannian manifold with bounded geometry, the estimate (\ref{eqn:SJ08_1}) is known by the work of Burkhardt-Guim~\cite{burkhardt2019pointwise}. 
    The key new ingredient of Theorem~\ref{thm:main} is that the continuous dependence (\ref{eqn:SJ08_1}) holds without the assumption of injectivity radius.

     \begin{cor}\label{cor:main_long_time}
      Suppose $\{g_t\}_{0 \leq t \leq T}$ is a Ricci flow on a complete manifold $M$ with bounded curvature 
      \begin{align}
          \sup_{M \times [0, T]}|Rm|_{g(t)} \leq \Lambda.    \label{eqn:SJ11_1}
      \end{align}
      Then there exist a small constant $\epsilon=\epsilon(n, T\Lambda)$ and a large constant $A=A(n, T\Lambda)$ with the following property.
      
         For each metric $\hat{g}_0$ satisfying $\displaystyle \sup_M |\hat{g}_0-g_0|_{g_0}<\epsilon$, there exists a Ricci-Deturck flow $\{g_t'\}_{0 \leq t \leq T}$, which uses $g_t$ as background and satisfies
      \begin{align}
           \sup_{M \times [0, T]} |\hat{g}_t-g_t|_{g_t} \leq A \sup_M |\hat{g}_0-g_0|_{g_0}. 
           \label{eqn:SJ11_2}
      \end{align}
     \end{cor}

    This Corollary can be regarded as an improvement of the uniqueness of Ricci flow solutions with bounded curvature (cf.~\cite{chen2006uniqueness}). 
    By dividing the time periods into $N$-pieces $\cup_{k=0}^{N-1} [\frac{kT}{N}, \frac{(k+1)T}{N}]$,  the proof of Corollary~\ref{cor:main_long_time} follows from applying (\ref{eqn:SJ08_1}) in Theorem~\ref{thm:main} repeatedly.\\

   Our results are based on the framework work of Koch-Lamm~\cite{koch2012geometric}.
   We sketch the basic steps and key difficulties as follows. \\

   \begin{proof}[Outline of the proof of Theorem~\ref{thm:main}:]

    We first study the short-time existence of the Ricci-Deturck flow.  
    For simplicity, we assume that the initial metric $g_0$ satisfies uniform curvature and curvature derivatives bound up to order $3$ (cf. (\ref{eqn:SI18_1})). More general cases will be discussed in Section~\ref{sec:final}.  Let $g_t$ be a family of the Ricci-Deturck flow, with background metric $\bar{g}=g_0$. Then $h_t=g_t-g_0$ satisfies the differential equation 
    \begin{align}
    (\partial_t-\Delta_L) h=-2Rc(g_0)
    +\cR[h] + \nabla^* \cS[h],    
    \label{eqn:SJ14_1}
    \end{align}
    where $\Delta_L$ is the Lichnerowicz Laplacian, $\cR[h]$ and $\cS[h]$ are quadratic terms of $h$ 
    and $\nabla h$. Let $K(x,t;y,s)=K(x,y,t-s): Sym^2(T_y^*M) \to Sym^2(T_x^*M)$ be the tensor heat kernel with respect to the metric $g_0$.  Then (\ref{eqn:SJ14_1}) is equivalent to the following integration equation 
    \begin{align}
    h(x,t)=\Phi(h)(x,t), 
    \label{eqn:SJ14_2}
    \end{align}
    where 
   \begin{align}
    \Phi(h)(x,t) := 
    \int_0^t \int_{M} \big\{ K(x,t;y,s) (-2Rc(g_0) + \cR [h])(y,s) + \nabla_y K(x,t;y,s) \cS [h](y,s) \big\} \, \dd y \, \dd s. 
   \label{eqn:SJ14_3} 
   \end{align}
   Let $X_T$  be a Banach space of some functions defined on $M \times [0, T]$, with norm $\|\cdot\|_{X_T}$ defined as follows. 
    \begin{align*}
      \|f\|_{X_T} := \sup_{x \in M, r \in (0, \sqrt{T}]} \left\{\| f \|_{L^\infty(P(x,r^2))}
         + \frac{\| \nabla f \|_{L^2(P(x,r^2))}}{|B(x,r)|^{\frac{1}{2}}}
         +r^{\frac{n+2}{n+4}} 
         \frac{\| \nabla f \|_{L^{n+4}(\Omega(x,r^2))}}{|B(x,r)|^{\frac{1}{n+4}}} \right\}, 
   \end{align*}
   where
   \begin{align}
    P(x,r^2) :=B(x,r) \times (0,r^2),  \quad \Omega(x,r^2) :=B(x,r) \times (\frac{r^2}{2}, r^2). 
   \label{eqn:SJ14_4} 
   \end{align}
   Note that the definition of $\|\cdot\|_{X_T}$-norm here is almost the same as that of Koch-Lamm, except for an extra normalization.  A solution of the integration equation (\ref{eqn:SJ14_2}) is exactly a fixed point of $\Phi$, whose existence is guaranteed by showing that $\Phi$ is a contraction mapping in a small ball in $X_T$. 
   For this purpose, as in~\cite{koch2012geometric}, we need to introduce another renormalized norm $\|\cdot\|_{Y_T}$. Then the contraction mapping property is realized by the tensor heat kernel estimates,  the interaction between $\|\cdot\|_{X_T}$ and $\|\cdot\|_{Y_T}$-norm, and the fact $T$ is very small. 
   Since there is no injectivity radius lower bound in our situation, a key difficulty is to set up  the uniform Calder\'on-Zygmund inequality (cf. discussion nearby (\ref{eqn:SJ15_1})):
   \begin{align*}
    \left  \|  \int_0^t\int_{M}  
     \nabla_x \nabla_y K (x, t;y,s)   \cS(y,s)\dd y \dd s \right \|_{L^{n+4}(\Omega(x_0, r^2))}
	\leq  C(n, T)  \|    \cS  \|_{L^{n+4}(\Omega(x_0, r^2))}. 
   \end{align*}
   After we obtain the short-time existence of the Ricci-Deturck flow, 
   a standard diffeomorphism action will transform the Ricci-Deturck flow solution to a Ricci flow solution.  Thus, we finish the proof of the short-time existence of the Ricci flow initiated from $g_0$.

   Then we show the continuous dependence (\ref{eqn:SJ08_1}). By shrinking $T$ slightly if necessary, we know the Ricci flow initiated from $g_0$ has bounded curvature $\Lambda$. Let $\hat{g}_0$ be a metric nearby $g_0$ in $C^0$-sense. Using $\{g_t\}_{0 \leq t \leq T}$ as background, the Ricci-Deturck flow solution $\{\hat{g}_t\}_{0 \leq t \leq T}$ initiated from $\hat{g}_0$ satisfies
   \begin{align}
       (\partial_t-\Delta_L) h= \cR[h] + \nabla^* \cS[h]   \label{eqn:SJ14_5}
   \end{align}
   where $h=h_t=\hat{g}_t-g_t$, and $\Delta_L$ is the Lichnerowicz Laplacian with respect to the metric $g_t$. 
   Similarly to the previous case, a solution $h$ of (\ref{eqn:SJ14_5}) is a fixed point of  $h=\Phi(h)$ where
   \begin{align}
       \Phi(h)(x, t) 
       :=
       \int_0^t \int_{M} \big\{ K(x,t;y,s)  \cR [h](y,s) + \nabla_y^{g_s} K(x,t;y,s) \cS [h](y,s) \big\} \, \dd g_s(y) \, \dd s. 
   \label{eqn:SJ14_6}    
   \end{align}
   We also show that $\Phi$ is a contraction mapping in a small ball in a Banach space $X_T$. 
   Now the choice of $\|\cdot\|_{X_T}$ needs one more step of renormalization, as the metrics and volume forms of $\{g_t\}_{0 \leq t \leq T}$ are evolving along the time, although in a uniform way.  
   Also, we need to construct Banach space $Y_T$ and develop heat kernel estimates to obtain the contraction property of $\Phi$.  An important difference here is that the current tensor heat kernel is for evolving metrics, whose gradient estimate need particular attention nearby the initial time. In other words, in the equation (\ref{eqn:SJ14_6}) above, 
   the estimate of $\nabla_y^{g_s} K(x,t;y,s)$ is important (cf. Proposition~\ref{prop:SI06_flow}).  Another key point is to obtain a uniform Calder\'on-Zygmund inequality (cf. the discussion nearby (\ref{eqn:SI01_18})): 
   \begin{align}
			\left\| \int_0^t\int_{M} \nabla_x^{g_t} \nabla_y^{g_s} K (x, t;y,s)\cS(y,s)\dd g_s(y) \dd s \right\|_{L^{n+4}(\Omega(x_0, r^2))}
            \leq C(n,\Lambda T)  \| \cS \|_{L^{n+4}(\Omega(x_0, r^2))},
    \label{eqn:SJ14_7}
	\end{align}
    for all $x_0 \in M$ and $r \in (0, \sqrt{T}]$. 
    Here,  $\Omega(x_0, r)=B_{g_0}(x_0, r) \times (\frac{r^2}{2}, r^2)$. 
    After a routine argument, we can obtain that $\Phi$ is a contraction mapping. Therefore, initiated from $\hat{g}_0$ which is close to $g_0$ in $C^0$-sense, there exist a Ricci-Deturck solution (with $\{g_t\}_{0 \leq t \leq T}$ as background) satisfying
    \begin{align*}
        \|h\|_{X_T} \leq C \|h_0\|_{L^{\infty}}, 
    \end{align*}
    which implies (\ref{eqn:SJ08_1}) by the definition of $\|\cdot\|_{X_T}$-norm. 

    It is clear that (\ref{eqn:SJ08_1}) yields the uniqueness of the Ricci-Deturck flow which uses 
    $\{g_t\}_{0 \leq t \leq T}$ as background metric and $g_0$ as initial metric. 
    If $\{\tilde{g}_t\}_{0 \leq t \leq T}$ is another Ricci flow solution initiated from $g_0$, we know
    $\varphi_t^*(\tilde{g}_t)$ satisfies the Ricci-Deturck flow solution, backgrounded by $\{g_t\}_{0 \leq t \leq T}$ and initiated from $g_0$. Thus $\varphi_t^*(\tilde{g}_t)\equiv g_t$
    for each $t \in [0, T]$, which in turn implies that $\varphi_t \equiv Id$. 
    Thus, $\tilde{g}_t \equiv g_t$. We arrive the uniqueness of the Ricci flow initiated from $g_0$.
    \end{proof}

    The paper is organized as follows. 
    In Section~\ref{sec:compute_RD_operator}, we study the linearization and difference of the Ricci-DeTurck operators.  In Section~\ref{sec:heatkernel}, we list the necessary tensor heat kernel estimates.  In Section~\ref{sec:exist}, we define the norms $X_T$ and $Y_T$ with respect to fixed background metrics.  Then we apply the contraction mapping principle in the Banach space $X_T$ to establish the existence and uniqueness of the Ricci-Deturck flow solutions.
    In Section~\ref{sec:cd}, we define the norms $X_T$ and $Y_T$ with respect to evolving background metrics. Then we repeat the process to study the perturbation of the Ricci flow and show the continuous dependence of the Ricci-Deturck flow with respect to the initial metric in the $C^0$-norm.  Finally, in Section~\ref{sec:final}, we finish the proof of the main theorem. 

\vspace{10pt}
\noindent
\textbf{Acknowledgments}
    The authors are supported by Project of Stable Support for Youth Team in
	Basic Research Field,  Chinese Academy of Sciences (YSBR-001), the National Natural Science Foundation of China (NSFC-12431003),  and a research fund from University of Science and Technology of China.

\section{The Ricci-DeTurck operator}
\label{sec:compute_RD_operator}
    
Let $g$ be a complete Riemannian metric on $M$ and ${\hat{g}}$ be a small perturbation of $g$.  Let $\bar{g}$ be a proper reference metric on $M$.
Then the Ricci-DeTurck operator is defined as
    \begin{align}
        P_{\bar{g}}(g) := -2 \Ric(g) - \cL_{X_{\bar{g}}(g)}g.
    \label{eqn:SA09_3}
    \end{align}
We shall study the perturbation of $P$. That is, we shall calculate the
quantity $P_{\bar{g}}(\hat{g})-P_{\bar{g}}(g)$. 

For simplicity of notation, we define
 \begin{align}
  h:=\hat{g}-g, \quad  h_{ij}:=\hat{g}_{ij} -g_{ij},  \quad
  u^{ij}:=\hat{g}^{ij}-g^{ij}.    \label{eqn:SB24_9}
 \end{align}
 We use $_{,i}$ to denote $\nabla_{i}^{g}$. In other words, all covariant derivatives are calculated with respect to metric $g$. 
 Then it is clear that 
 \begin{align}
    u^{pl}=-g^{pk}g^{ls}h_{ks} -u^{pk}g^{ls}h_{ks}, \quad
    u_{,i}^{pl}=-h_{ks,i} \hat{g}^{ls}\hat{g}^{pk}. 
 \label{eqn:SB21_3}   
 \end{align}

 Define
 \begin{align}
     \boldsymbol{\gamma}_{ijl} :=\frac12(h_{il,j}+h_{jl,i}-h_{ij,l}), \quad
     \mathcal{B}_l :=g^{ij} \boldsymbol{\gamma}_{ijl}, \quad
     \vartheta^l :=g^{lk}\mathcal{B}_k=g^{lk}g^{ij} \boldsymbol{\gamma}_{ijk}.
 \label{eqn:SB24_3}    
 \end{align}
 Then we have
 \begin{align}
     \hat{\gamma}_{ij}^k& :=\hat{\Gamma}_{ij}^k-\Gamma_{ij}^k
     =\frac12 \hat{g}^{kl} (h_{li,j}+h_{lj,i}-h_{ij,l})
     =\hat{g}^{kl} \boldsymbol{\gamma}_{ijl}
     =g^{kl}\boldsymbol{\gamma}_{ijl} + u^{kl}\boldsymbol{\gamma}_{ijl}. 
     \label{eqn:SB16_3}  
 \end{align}
 The direct calculation shows that
 \begin{align}
     -2(\hat{R}_{ij}-R_{ij}) 
     =\Delta_L h_{ij} - \mathcal{L}_{\vartheta}g
      +   \cQ _{Rc}(h),    \label{eqn:DJ01_3}
 \end{align}
 where we used the definition
 \begin{align}
      \cQ _{Rc}(h)&:=-2u^{lq}(\boldsymbol{\gamma}_{ijq,l}-\boldsymbol{\gamma}_{ilq,j})
      \underbrace{
      -2(u_{,l}^{lq}\boldsymbol{\gamma}_{ijq} -u_{,j}^{lq}\boldsymbol{\gamma}_{ilq}) 
      -2\hat{\gamma}_{lp}^l\hat{\gamma}_{ij}^p + 2\hat{\gamma}_{jp}^l \hat{\gamma}_{il}^p}_{I}.   \label{eqn:DJ01_1}
 \end{align}

 Define
 \begin{align*}
     X^k :=g^{ij}\bar{\gamma}_{ij}^k :=g^{ij} (\bar{\Gamma}_{ij}^k-\Gamma_{ij}^k), \quad  
     \hat{X}^k :=\hat{g}^{ij}(\bar{\Gamma}_{ij}^k-\hat{\Gamma}_{ij}^k)
     =\hat{g}^{ij}(\bar{\gamma}_{ij}^k-\hat{\gamma}_{ij}^k).   
 \end{align*}
 Note that $\bar{\gamma}_{ij}^k=\bar{\Gamma}_{ij}^k-\Gamma_{ij}^k$ is a tensor. 
 Let $Y=\hat{X}-X$. Then the direct calculation shows that
 \begin{align}
   Y^k&=\hat{X}^k-X^k=-\vartheta^k 
   -(h \circ \bar{\gamma})^k  
   \underbrace{-u^{kl} \mathcal{B}_l-u^{ij}(\hat{\gamma}_{ij}^k+g^{lq}h_{jq}\bar{\gamma}_{il}^k)}_{-G^k}, 
 \label{eqn:SB16_1}
 \end{align}
 where 
 \begin{align*}
     h \circ \bar{\gamma}
     =(h \circ \bar{\gamma})^k \frac{\partial}{\partial x^k}
     =g^{ip}g^{jq}h_{pq}\bar{\gamma}_{ij}^k \frac{\partial}{\partial x^k}. 
 \end{align*}
 Thus
 \begin{align}
   &\mathcal{L}_{\hat{X}}\hat{g} -\mathcal{L}_{X}g=\mathcal{L}_{X}h-\mathcal{L}_{\vartheta} g -\mathcal{L}_{h \circ \bar{\gamma}} g- \cQ _{\mathcal{L}}(h),
 \label{eqn:SB19_1}
 \end{align}
 where 
 \begin{align}
    \cQ _{\mathcal{L}}(h) :=\mathcal{L}_{G}g +\mathcal{L}_{\vartheta}h +\mathcal{L}_{G}h +\mathcal{L}_{h \circ \bar{\gamma}}h=\mathcal{L}_{G} \hat{g} +\mathcal{L}_{\vartheta}h+\mathcal{L}_{h \circ \bar{\gamma}} h.  \label{eqn:DJ01_7}
 \end{align}
We use the notation $P$ and $\hat P$ to represent $P_{\bar g} (g)$ and $ P_{\bar g}(\hat g) $, respectively. Combining (\ref{eqn:DJ01_3}) and (\ref{eqn:SB19_1}) together, we have
\begin{align}
    \hat{P}-P=\underbrace{\Delta_L h-\mathcal{L}_{X}h+\mathcal{L}_{h \circ \bar{\gamma}}g}_{L(h)} +\underbrace{ \cQ _{Rc}(h)+ \cQ _{\mathcal{L}}(h)}_{ \cQ (h)}
    =L(h)+ \cQ (h). 
\label{eqn:SB24_4}    
\end{align}

Our next step is to estimate $ \cQ (h)$. 
It follows from (\ref{eqn:DJ01_7}) that
\begin{align}
 ( \cQ _{\mathcal{L}})_{ij}
 =G_{,i}^k\hat{g}_{kj} +G_{,j}^k\hat{g}_{ki}
  +\left\{\vartheta_{,i}^k + (h \circ \bar{\gamma})_{,i}^k \right\} h_{kj} 
  +\left\{ \vartheta_{,j}^k+ (h \circ \bar{\gamma})_{,j}^k \right\} h_{ki}
  -Y^k h_{ij,k}. 
\label{eqn:SB18_4}
\end{align}
By (\ref{eqn:SB16_1}), we have
\begin{align*}
    G_{,i}^k=(u^{kl}\mathcal{B}_{l,i}+u^{pq}\hat{\gamma}_{pq,i}^k)+  
    \left\{u_{,i}^{kl}\mathcal{B}_l +u_{,i}^{pq}(\hat{\gamma}_{pq}^k+g^{lm}h_{qm}\bar{\gamma}_{pl}^k) +u^{pq}g^{lm}(h_{qm}\bar{\gamma}_{pl}^k)_{,i} \right\},
\end{align*}
where the first two terms in the last line consist of second derivatives of the tensor $h$. By straightforward calculation, they can be simplified as
\begin{align*}
  u^{kl}\mathcal{B}_{l,i}+u^{pq}\hat{\gamma}_{pq,i}^k
    =(\hat{g}^{kl}\hat{g}^{pq}-g^{kl}g^{pq})\boldsymbol{\gamma}_{pql,i} 
    +u^{pq}u_{,i}^{kl} \boldsymbol{\gamma}_{pql}. 
\end{align*}
According to the order of derivatives of $h$, we obtain the decomposition 
\begin{align*}
 G_{,i}^k
 =(\hat{g}^{kl}\hat{g}^{pq}-g^{kl}g^{pq})\boldsymbol{\gamma}_{pql,i}
 +\{\underbrace{u_{,i}^{kl}\mathcal{B}_l+u_{,i}^{pq}(\hat{\gamma}_{pq}^k+g^{lm}h_{qm}\bar{\gamma}_{pl}^k) +u^{pq}g^{lm}(h_{qm}\bar{\gamma}_{pl}^k)_{,i}
 +u^{pq}u_{,i}^{kl} \boldsymbol{\gamma}_{pql}}_{A_i^k}
 \}.
\end{align*}
Plug this into (\ref{eqn:SB18_4}). 
After a lengthy, but routine calculation, we obtain 
\begin{align*}
( \cQ _{\mathcal{L}})_{ij}&=u^{pq}\boldsymbol{\gamma}_{pqj,i}+u^{pq}\boldsymbol{\gamma}_{pqi,j}\\
&\quad+\underbrace{(A_i^k\hat{g}_{kj}+A_j^k\hat{g}_{ki}-Y^k h_{ij,k}) +g^{pq}g^{ms}\left( (h_{sq}\bar{\gamma}_{mp}^k)_{,i}h_{kj} +(h_{sq}\bar{\gamma}_{mp}^k)_{,j}h_{ki} \right)}_{II}.    
\end{align*}
Combining this with (\ref{eqn:DJ01_1}) yields that
\begin{align*}
  \cQ _{ij} -I-II
 =( \cQ _{\mathcal{L}})_{ij}-I+( \cQ _{Rc})_{ij}-II
 =u^{pq}(\boldsymbol{\gamma}_{pqj,i}+\boldsymbol{\gamma}_{pqi,j}+2\boldsymbol{\gamma}_{ipq,j}-2\boldsymbol{\gamma}_{ijq,p}). 
\end{align*}
Using the symmetry of $u^{pq}$, the direct calculation shows that 
\begin{align*}
    &\quad u^{pq} \left(\boldsymbol{\gamma}_{pqi,j}+\boldsymbol{\gamma}_{pqj,i}+2\boldsymbol{\gamma}_{ipq,j}-2\boldsymbol{\gamma}_{ijq,p}\right) \\
    &=u^{pq}h_{ij,pq} + \underbrace{u^{pq} \left(R_{pjq}^sh_{si}+R_{pji}^sh_{sq}+ R_{piq}^sh_{sj}+R_{pij}^sh_{sq} +\frac12(R_{ijp}^sh_{sq}+R_{ijq}^s h_{sp}) \right)}_{III}. 
\end{align*}
We define a $(1,2)$-tensor $\cS$ by
\begin{align}
    \cS_{ij}^k :=-u^{kl}h_{ij,l}.      \label{eqn:SB24_2}
\end{align}
Then we have the decomposition of $ \cQ $ as
\begin{align*}
     \cQ _{ij}=(\nabla^* \cS)_{ij} +I+II+III-u_{,q}^{qp}h_{ij,p},  
\end{align*}
where 
\begin{align}
    I&=-2(u_{,l}^{lq}\boldsymbol{\gamma}_{ijq} 
      -u_{,j}^{lq}\boldsymbol{\gamma}_{ilq}) 
      -2\hat{\gamma}_{lp}^l\hat{\gamma}_{ij}^p + 2\hat{\gamma}_{jp}^l \hat{\gamma}_{il}^p, \label{eqn:SB24_5}\\
    II&=\left( u_{,i}^{kl}\mathcal{B}_l+u_{,i}^{pq}(\hat{\gamma}_{pq}^k+g^{lm}h_{qm}\bar{\gamma}_{pl}^k) +u^{pq}g^{lm}(h_{qm}\bar{\gamma}_{pl}^k)_{,i}
 +u^{pq}u_{,i}^{kl} \boldsymbol{\gamma}_{pql} \right) \hat{g}_{kj} \notag\\
    &\quad+\left( u_{,j}^{kl}\mathcal{B}_l+u_{,j}^{pq}(\hat{\gamma}_{pq}^k+g^{lm}h_{qm}\bar{\gamma}_{pl}^k) +u^{pq}g^{lm}(h_{qm}\bar{\gamma}_{pl}^k)_{,j}
 +u^{pq}u_{,j}^{kl} \boldsymbol{\gamma}_{pql} \right) \hat{g}_{ki} \notag\\
    &\quad + \left(\vartheta^k 
   +(h \circ \bar{\gamma})^k+u^{kl} \mathcal{B}_l  
   +u^{pq}(\hat{\gamma}_{pq}^k+g^{lm}h_{qm}\bar{\gamma}_{pl}^k) \right) h_{ij,k}\notag\\
   &\quad+g^{pq}g^{ms}\left( (h_{sq}\bar{\gamma}_{mp}^k)_{,i}h_{kj} +(h_{sq}\bar{\gamma}_{mp}^k)_{,j}h_{ki} \right), \label{eqn:SB24_6}\\
   III&=u^{pq} \left(R_{pjq}^sh_{si}+R_{pji}^sh_{sq}+ R_{piq}^sh_{sj}+R_{pij}^sh_{sq} +\frac12(R_{ijp}^sh_{sq}+R_{ijq}^s h_{sp}) \right).     \label{eqn:SB24_7}
\end{align}

In summary, we have
\begin{prop}\label{prop:RD_operator_compute}
    Suppose $\bar{g}, g$ and $\hat{g}$ are three Riemannian metrics on $M$.
    Let $h=\hat{g}-g$. Then we have
    \begin{align*}
        P_{\bar{g}}(\hat{g}) -P_{\bar{g}}(g)=Lh+ \cQ [h]. 
    \end{align*}
    Here $L$ is a linear operator defined as
    \begin{align}
        L h=\Delta_L h-\mathcal{L}_{X}h+\mathcal{L}_{h \circ \bar{\gamma}}g= \Delta h +  E \ast h + F \ast \nabla h,    \label{eqn:SB24_8} 
    \end{align}
    where $E \ast h$ and $F \ast h$ have the following expression in normal coordinates: 
    \begin{align}
        &(E \ast h)_{ij} := (R_{ipqj} + R_{jpqi} + \bar \gamma_{pqj,i} + \bar \gamma_{pqi,j})h_{pq}  -  h_{ip}(R_{pj} + \bar \gamma_{mmp,j}) - h_{jp}(R_{pi} + \bar \gamma_{mmp,i}), \label{eqn:SC16_3} \\
        &(F \ast \nabla h)_{ij} := h_{pq,i} \bar \gamma_{pqj} + h_{pq,j} \bar \gamma_{pqi} - \bar \gamma_{mmk}h_{ij,k}. 
        \label{eqn:SC16_4}
    \end{align}
    The term $ \cQ [h]$ is a quadratic term of $h$ satisfying
    \begin{align}
    \cQ [h]={}& \nabla ^\ast \cS [h] + \cR [h],    \label{eqn:SC16_1}
    \end{align}
    where
    \begin{align}
        \cS [h]_{ij}^k :=\cS_{ij}^k ,\quad  
        \cR [h]_{ij} :=-u_{,q}^{qp}h_{ij,p}+I+II+III,
        \label{eqn:SC16_2}
    \end{align}
    where $S$ is defined in (\ref{eqn:SB24_2}), $u^{pq}$ is defined in (\ref{eqn:SB24_9}), the term $I$, $II$ and $III$ are defined in (\ref{eqn:SB24_5}), (\ref{eqn:SB24_6}) and (\ref{eqn:SB24_7}) respectively.
\end{prop}

In a particular case, where $\bar{g}=g$, we have $\bar{\gamma}_{ij}^k=\bar{\Gamma}_{ij}^k-\Gamma_{ij}^k=0$. Then 
\begin{align}
    Lh=\Delta h + 2R_{ipqj} h_{pq} -R_{ip}h_{pj} -R_{jp}h_{pi}=\Delta_L h,
\label{eqn:SL07_9}
\end{align}
where $\Delta_L$ is the Lichnerowicz Laplacian.

\section{Heat kernel estimates} 
\label{sec:heatkernel}

In this paper, we let $g_t \equiv g$, or let $g_t$ evolve by the Ricci flow equation. 
    The evolving manifold is denoted by $\{(M, g_t), 0 \leq t \leq 1\}$. We always assume the  following conditions to hold. 
    \begin{itemize}
        \item  If $g_t \equiv g$, we assume
        \begin{align}
            \sum_{k=0}^3 |\nabla^k Rm| \leq 1. 
        \label{eqn:SI18_1}    
        \end{align}
        \item  If $g_t$ evolves by Ricci flow, we assume
        \begin{align}
             |Rm| \leq 1. 
        \label{eqn:SI18_2}    
        \end{align}          
    \end{itemize}
    Note that in (\ref{eqn:SI18_1}) and (\ref{eqn:SI18_2}), the constant $1$ is chosen for simplicity. If the left hand side quantities are bounded by some constant $\Lambda_0$,  we can always do rescaling and let the right hand side be $1$.  
    For the Ricci flow solution, (\ref{eqn:SI18_2}) actually implies (\ref{eqn:SI18_1}) on the time interval  $[\frac12, 1]$ by Shi-type curvature estimates and rescaling. 
    
    We also use $g(t)$ to denote $g_t$ to avoid misunderstanding of taking $t$-derivatives. 
    It is clear that the constant $1$ can be replaced by other uniform constants $T_0$ and $\Lambda_0$. Then our discussion also works for evolving metrics on $M \times [0, T_0]$ with uniform curvature bounds $\Lambda_0$. However, for simplicity of notation, we prefer to fix both $T_0$ and $\Lambda_0$ to be $1$.

	\begin{defn}
    \label{dfn:SL06_3}
		For each $x,y\in M$ and $0 \leq s<t \leq 1$, we denote $K(x,t;y,s)$ as a map
        \begin{align*}
            K(x,t;y,s): Sym^2(T_y^* M ) \to Sym^2(T_x^* M).
        \end{align*}
        We say that $K(x,t;y,s)$ is a heat kernel of the heat operator $(\partial_t-\Delta_L)$ if it satisfies
			\begin{align}
				&(\partial_t - \Delta_L) K(x,t;y,s) = 0, \label{eqn:SI06_15}\\
				&\lim_{t\to s^+ } K(x,t; y,s) =  \delta_y Id_{Sym^2(T^*M)}.
                 \label{eqn:SI06_16}
			\end{align}
	\end{defn}

    The identity (\ref{eqn:SI06_16}) can be understood as follows. 
    Fix $y$. For each smooth map 
    \begin{align*}
        F(x,t): Sym^2(T_x^*M) \to Sym^2(T_y^*M), 
    \end{align*}
    we have
    \begin{align*}
        \lim_{t \to s^{+}} \int_M F(x,t)K(x,t;y,s)d\mu_x(s)=F(y,s). 
    \end{align*}
    The standard duality implies that 
    \begin{align}
       &(-\partial_s -\Delta_L + H) K(x,t;y,s)=0,      \label{eqn:SI06_4}\\
       &\lim_{s \to t^{-} } K(x,t; y,s) =  \delta_x \text{Id}_{Sym^2(T^*M)}, 
    \end{align}
    where $H=-\frac12 tr_{g(s)}(\partial_s g)$.  In the static metric case, $H=0$. 
    In the Ricci flow setting, $H=-R$ where $R$ is the scalar curvature.

    Now we discuss the heat kernel gradient in the Ricci flow setting. 
    The static metric case is simpler, and we leave it to the reader. 
    Suppose $u$ is a smooth $(0,2)$-tensor field satisfying 
    \begin{align*}
        (\partial_t-\Delta_L)u=0 
    \end{align*}
    along the Ricci flow. Then the direct calculations show that
    \begin{align*}
       &\quad \{ (\partial_t-\Delta) \nabla u \}_{ij,k}\\
       &=2\left\{R_{jpqk}u_{ip,q} +R_{ipqk}u_{pj,q} +R_{ipqj}u_{pq,k}\right\}
     -\left\{R_{ip}u_{pj,k}+R_{jp}u_{ip,k}+R_{kq}u_{ij,q} \right\}
     +2R_{ipqj,k}u_{pq}. 
    \end{align*}
    In short, we have
    \begin{align}
        (\partial_t - \Delta ) \nabla u=Rm*(\nabla u) + (\nabla Rm) * u.   \label{eqn:SI07_1}
    \end{align}
    Similarly, if $w$ is a smooth $(0,2)$-tensor field satisfying $(-\partial_s-\Delta_L +R)w=0$, then we have
    \begin{align}
         (-\partial_s - \Delta) \nabla w=Rm*(\nabla w) + (\nabla Rm) * w.  \label{eqn:SI07_2}
    \end{align}
    We can continue to take space derivatives. Note that for heat kernel $K(x,t;y,s)$, the $x$-derivative $\nabla_x^{g(t)}$ and $y$-derivative $\nabla_y^{g(s)}$ are different. 

    \begin{defn}
    \label{dfn:SI07_1}
    Note that 
    \begin{align}
        \nabla_x^{g(t)} \nabla_y^{g(s)}K(x,t;y,s): 
        Sym^2(T_y^* M ) \otimes T_y^*M \to Sym^2(T_x^* M) \otimes T_x^*M.   \label{eqn:SI06_5}
    \end{align}
    For simplicity of notation, we define
    \begin{align}
        \mathcal{K}(x,t;y,s) := \nabla_x^{g(t)} \nabla_y^{g(s)}K(x,t;y,s). \label{eqn:SI06_6}
    \end{align}
	\end{defn}

    \begin{prop}
   \label{prop:SI06_0}
    Let $(M, g)$ be a complete Riemannian manifold satisfying (\ref{eqn:SI18_1}). Then the following heat kernel estimates hold for all $x \in M$ and $0<s<t\leq 1$.
    \begin{align}
				\big\{|K| +\sqrt{t-s}\left(|\nabla_x K| 
                +|\nabla_y K|\right) \big\}(x,t;y,s)
                <\frac {C}{|B(x,\sqrt{t-s})|} \cdot \exp \left ( -\frac {d^2(x,y)}{4D(t-s)} \right ).
                \label{eqn:SI06_7}
	\end{align}
    Let $\mathcal{K}=\mathcal{K}(x,t;y,s)=\nabla_x \nabla_y K(x,t;y,s)$. It satisfies 
    \begin{align}
				\big\{|\mathcal{K}| +\sqrt{t-s}|\nabla_x \mathcal{K}| 
                +(t-s)|\partial_t \mathcal{K}| \big\}
                <\frac{1}{(t-s)} \cdot \frac {C}{|B(x,\sqrt{t-s})|} \cdot \exp \left ( -\frac {d^2(x,y)}{4D(t-s)} \right ).
                \label{eqn:SI06_8}
	\end{align}
    Here, both $C$ and $D$ are dimensional constants. 
   \end{prop}

   \begin{proof}
 For a fixed Riemannian metric with bounded curvature, standard heat kernel estimate implies 
       \begin{align}
           |K|(x,t;y,s)
          <\frac {C}{|B(x,\sqrt{t-s})|^{\frac12} \cdot|B(y,\sqrt{t-s})|^{\frac12}} 
          \cdot \exp \left ( -\frac {d^2(x,y)}{4D(t-s)} \right ). \label{eqn:SI06_9}
       \end{align}
       Volume comparison shows that 
       \begin{align*}
           \frac{1}{|B(y,\sqrt{t-s})|^{\frac12}} \leq \frac{1}{|B(x,\sqrt{t-s})|^{\frac12}} 
           \cdot e^{C\{d(x,y)+\sqrt{t-s}\}}.
       \end{align*}
       Plugging it into (\ref{eqn:SI06_9}) and adjusting the constant $D$ slightly if necessary, we obtain
       \begin{align}
         |K|(x,t;y,s)
         <\frac {C}{|B(x,\sqrt{t-s})|} \cdot \exp \left ( -\frac {d^2(x,y)}{4D(t-s)} \right ),
       \end{align}
       which provides the estimate of the first term in (\ref{eqn:SI06_7}).  The second and third terms can be estimated using the standard gradient estimate.  Thus, (\ref{eqn:SI06_7}) is proved. 

       Fixing $(y,s) \in M \times [0,1]$ and $V \in T_y^*M$, we can regard $\mathcal{K}(x,t;\cdot, \cdot)$ as $\nabla u$
       for some $(0,2)$-tensor field 
       $u=u_{ij}dx^idx^j$ satisfying the heat equation $(\partial_t-\Delta_L) u=0$. 
       Then $\nabla u$ satisfies (\ref{eqn:SI07_1}). 
       Note that now the metric is fixed with  $|Rm|+\sqrt{t}|\nabla Rm|$ uniformly bounded. 
       The standard heat solution estimate implies 
        \begin{align}
				|\mathcal{K}|(x,t;y,s)
                <\frac{1}{(t-s)} \cdot \frac {C}{|B(x,\sqrt{t-s})|} \cdot \exp \left ( -\frac {d^2(x,y)}{4D(t-s)} \right ).
                \label{eqn:SI06_10}
	   \end{align}
       Fixing $(y,s)$, we can regard $\nabla \mathcal{K}:=\nabla_x \mathcal{K}$ as a $(0,4)$-tensor field that satisfies
       \begin{align*}
           (\partial_t -\Delta) \nabla \mathcal{K}=Rm*(\nabla \mathcal{K}) 
           +(\nabla Rm)* \mathcal{K} + (\nabla \nabla Rm)* \mathcal{K}. 
       \end{align*}
       As $|Rm|+\sqrt{t}|\nabla Rm|+t|\nabla \nabla Rm|$ is uniformly bounded, 
       we can apply a standard gradient estimate for the heat-type solution again and obtain 
       \begin{align}
				\sqrt{t-s}|\nabla_x \mathcal{K}|(x,t;y,s)
                <\frac{1}{(t-s)} \cdot \frac {C}{|B(x,\sqrt{t-s})|} \cdot \exp \left ( -\frac {d^2(x,y)}{4D(t-s)} \right ).
                \label{eqn:SI06_13}
	   \end{align}
       From this we regard 
       $\nabla_x \nabla_x \mathcal{K}$ as a $(0,5)$-tensor field satisfying some heat-type solution, with lower order term coefficient bounded by $\sum_{k=0}^3 t^{\frac{k}{2}}|\nabla^k Rm|$. 
       The same gradient estimate implies that
        \begin{align}
				(t-s)|\nabla_x \nabla_x \mathcal{K}|(x,t;y,s)
                <\frac{1}{(t-s)} \cdot \frac {C}{|B(x,\sqrt{t-s})|} \cdot \exp \left ( -\frac {d^2(x,y)}{4D(t-s)} \right ).
                \label{eqn:SI06_14}
	   \end{align}
 Taking the time derivative of (\ref{eqn:SI07_1}), we see that $\partial_t \mathcal{K}$ satisfies a heat-type equation with lower-order-term coefficients bounded by $\sum_{k=0}^3 t^{\frac{k}{2}}|\nabla^k Rm|$. 
       As $\Delta= tr(\nabla \nabla)$, we can use $|\nabla_x \nabla_x \mathcal{K}|$ and lower order terms to bound $\partial_t \mathcal{K}$.  It follows from (\ref{eqn:SI06_14}), (\ref{eqn:SI06_13}) and (\ref{eqn:SI06_10}) that 
        \begin{align}
				(t-s)|\partial_t \mathcal{K}|(x,t;y,s)
                <\frac{1}{(t-s)} \cdot \frac {C}{|B(x,\sqrt{t-s})|} \cdot \exp \left ( -\frac {d^2(x,y)}{4D(t-s)} \right ).
                \label{eqn:SI06_144}
	   \end{align}
 Combining the previous four inequalities, we arrive at (\ref{eqn:SI06_8}). 
   \end{proof}

   \begin{prop}
   \label{prop:SI06_flow}
    Let $\{(M, g_t), 0 \leq t \leq 1\}$ be a complete Ricci flow solution satisfying (\ref{eqn:SI18_2}). 
    Then the following heat kernel estimates hold for all $x \in M$ and $0<s<t\leq 1$.
    \begin{align}
				\big\{|K| +\sqrt{t-s}\left(|\nabla_x^{g_t} K| 
                +|\nabla_y^{g_s} K|\right) \big\}(x,t;y,s)
                <\frac {C}{|B(x,\sqrt{t-s})|} \cdot \exp \left ( -\frac {d^2(x,y)}{4D(t-s)} \right ).
                \label{eqn:SI06_7flow}
	\end{align}
    Let $\mathcal{K}=\mathcal{K}(x,t;y,s)=\nabla_x^{g(t)} \nabla_y^{g(s)} K(x,t;y,s)$. It satisfies 
    \begin{align}
				\big\{|\mathcal{K}| +\sqrt{t-s}|\nabla_x^{g_t} \mathcal{K}| 
                +(t-s)|\partial_t \mathcal{K}| \big\}
                <\frac{1}{(t-s)} \cdot \frac {C}{|B(x,\sqrt{t-s})|} \cdot \exp \left ( -\frac {d^2(x,y)}{4D(t-s)} \right)
                \label{eqn:SI06_8flow}
	\end{align}
    whenever $\frac14<s<t<1$.
    Here, both $C$ and $D$ are dimensional constants. 
   \end{prop}

   \begin{proof}
       Since $|Rm|$ is uniformly bounded on the space-time $M \times [0,1]$, by the Shi-type estimate, we know that $t^{\frac{k}{2}}|\nabla^k Rm|$ is uniformly bounded by $C_k$ for each positive integer $k$.  In particular, on the time interval $[\frac18, 1]$, all covariant curvature derivatives are uniformly bounded.  This implies that (\ref{eqn:SI06_8}) and most of (\ref{eqn:SI06_7flow}).   The only missing thing is the following part of (\ref{eqn:SI06_7flow}).
       \begin{align}
        |\nabla_y^{g_s} K|(x,t;y,s)
        <\frac {C}{|B(x,\sqrt{t-s})|} \cdot \exp \left( -\frac {d^2(x,y)}{4D(t-s)} \right).
       \label{eqn:SI07_3} 
       \end{align}
       If we fix $(x,t)$ and set $w(y,s)=\nabla_y^{g_s}K(x,t;y,s)$, then $w$ satisfies equation (\ref{eqn:SI07_2}).   However, as $s \to 0$, we do not have a uniform bound for $\nabla Rm$. 
       The good news is that $|\nabla Rm|<C s^{-\frac12}$, which is integrable and suffices to prove (\ref{eqn:SI07_3}). The details are listed in the proof of Lemma~\ref{lem:SI06_1}.  
       Therefore, Lemma~\ref{lem:SI06_1} implies (\ref{eqn:SI07_3}). 
       Consequently, the proof of (\ref{eqn:SI06_7flow}) is complete. 
   \end{proof}

  \begin{lem}
   \label{lem:SI06_1}
    Let $\{(M, g_t), 0 \leq t \leq 1\}$ be a complete Ricci flow solution satisfying (\ref{eqn:SI18_2}).
    Let $w$ be a $(0,2)$-tensor field satisfying 
    \begin{align*}
        (-\partial_t - \Delta_L + R) w=0. 
    \end{align*}
    Then we have estimate
    \begin{align*}
        |\nabla w|(x,t) <  C \sup_{M \times [0,1]} |w|, \quad \forall \; x \in M, \; t \in [0, \frac12]. 
    \end{align*}
   \end{lem}

   \begin{proof}
       Let $\tau=1-t$.
       Thus, $w=w_{ij}dx^idx^j$ satisfies the differential equation 
       \begin{align*}
           (\partial_{\tau}-\Delta_L +R) w_{ij}=0. 
       \end{align*}
       Taking derivative in space direction implies 
       \begin{align*}
           \nabla_k \partial_{\tau} w_{ij} - \nabla_k \Delta_L w_{ij} + R_k w_{ij} + Rw_{ij,k}=0. 
       \end{align*}
       Since
       \begin{align*}
           (\partial_{\tau} \nabla w)_{ijk}=\nabla_k (\partial_{\tau} w)_{ij}
             -(R_{qj,k}+R_{qk,j}-R_{jk,q})w_{iq}
             -(R_{qi,k}+R_{qk,i}-R_{ik,q})w_{jq}, 
       \end{align*}
       we have
       \begin{align*}
           &\quad (\partial_{\tau}-\Delta) w_{ij,k} +(R_{qj,k}+R_{qk,j}-R_{jk,q})w_{iq}
             +(R_{qi,k}+R_{qk,i}-R_{ik,q})w_{jq}\\
           &=- R_k w_{ij} -R w_{ij,k}
             +[\nabla_k \Delta_L - \Delta \nabla_k] w_{ij}. 
       \end{align*}
       Note that
       \begin{align*}
         &\quad [\nabla_k \Delta_L - \Delta \nabla_k] w_{ij}\\
         &=2R_{ipqj,k}w_{pq} - (R_{kq,i}+R_{iq,k}-R_{ki,q})w_{qj} - (R_{kq,j}+R_{jq,k}-R_{kj,q})w_{qi}\\
         &\quad +2\left( R_{kpqi}w_{qj,p} 
       +R_{kpqj}w_{iq,p} +R_{ipqj}w_{pq,k}\right) -\left(R_{ip}w_{pj,k} +R_{jp}w_{pi,k}+R_{kq}w_{ij,q} \right). 
       \end{align*}
       It follows that 
       \begin{align*}
           &\quad (\partial_{\tau}-\Delta) w_{ij,k}\\
           &=2R_{ipqj,k}w_{pq} - 2(R_{kq,i}+R_{iq,k}-R_{ki,q})w_{qj} - 2(R_{kq,j}+R_{jq,k}-R_{kj,q})w_{qi} -R_{k}w_{ij}\\
         &\quad -R w_{ij,k} +2\left( R_{kpqi}w_{qj,p} 
       +R_{kpqj}w_{iq,p} +R_{ipqj}w_{pq,k}\right) -\left(R_{ip}w_{pj,k} +R_{jp}w_{pi,k}+R_{kq}w_{ij,q} \right). 
       \end{align*}
      In short,  we have
      \begin{align*}
          &\quad (\partial_{\tau}-\Delta) |\nabla w|^2\\
          &=2 \langle (\partial_{\tau}-\Delta) \nabla w, \nabla w\rangle -2|\nabla \nabla w|^2
          -2(R_{ip}w_{ij,k}w_{pj,k}+R_{jq}w_{ij,k}w_{iq,k}+R_{kl}w_{ij,k}w_{ij,l}) \\
         &\leq C(|\nabla Rm||w||\nabla w| +|Rm||\nabla w|^2), 
      \end{align*}
      which implies
       \begin{align*}
          (\partial_{\tau}-\Delta) |\nabla w| \leq  C(|\nabla Rm||w| +|Rm||\nabla w|).
       \end{align*}
       On the Ricci flow spacetime with $0 \leq t \leq \frac12$, we know $|Rm| \leq 1$ and $|\nabla Rm|<\frac{C}{\sqrt{t}}$. 
        As $\int_0^{\frac12} \frac{1}{\sqrt{t}} dt$ is finite, we can obtain the desired gradient estimate of $w$ from the above inequality and the maximum principle. 
   \end{proof}

\section{Short-time existence and uniqueness}
\label{sec:exist}

In this section, we study the short-time existence of a Ricci-Deturck flow initiated from a metric $g_0$ satisfying (\ref{eqn:SI18_1}). 
We also set
\begin{align}
    \bar{g}=g_0.    \label{eqn:SI1_2}
\end{align}

Using $\bar{g}$ as the default metric and $d\bar{g}$ as the default volume form, we follow Koch-Lamm~\cite{koch2012geometric} to define two functional spaces and their related norms.  

The localized norm $X_{B(x,r)}$ for $x \in M$ and $0 < r^2 < T$ is defined as follows:
\begin{align}
        \|f\|_{X_{B(x,r)}} &:=\| f \|_{L^\infty(P(x,r^2))}
         + \frac{\| \nabla f \|_{L^2(P(x,r^2))}}{|B(x,r)|^{\frac{1}{2}}}
         +r^{\frac{n+2}{n+4}} 
         \frac{\| \nabla f \|_{L^{n+4}(\Omega(x,r^2))}}{|B(x,r)|^{\frac{1}{n+4}}},
        \label{1def:local_X}
\end{align}
where
\begin{align}
    P(x,r^2) &:=B(x,r) \times (0,r^2), \label{eqn:SA11_3}\\
    \Omega(x,r^2) &:=B(x,r) \times (\frac{r^2}{2}, r^2). \label{eqn:SA11_4}
\end{align}

We define the localized norm $Y_{B(x,r)}$ for $x \in M$ and $0 <r^2< T$ in steps.
First, we define
\begin{subequations}
    \begin{align*}
            \| f_0 \|_{Y^0_{B(x,r)}} 
            &:=  \frac{\| f_0 \|_{L^1(P(x,r^2))}}{|B(x,r)|} 
            +r^{\frac{2n+4}{n+4}}  
            \frac{\| f_0 \|_{L^{\frac{n+4}{2}}(\Omega(x,r^2))}}{|B(x,r)|^{\frac{2}{n+4}}},\\
             \| f_1 \|_{Y^1_{B(x,r)}} &:= 
             \frac{\| f_1 \|_{L^2(P(x,r^2))}}{|B(x,r)|^{\frac{1}{2}}} 
            +r^{\frac{n+2}{n+4}}  \frac{\| f_1 \|_{L^{n+4}(\Omega(x,r^2))}}{|B(x,r)|^{\frac{1}{n+4}}}.
    \end{align*}
\end{subequations}
Then we set 
 \begin{align}
            \| f \|_{Y_{B(x,r)}} &:=\inf_{f = f_0 + \nabla^\ast f_1} \left\{ \| f_0 \|_{Y^0_{B(x,r)}} + \| f_1 \|_{Y^1_{B(x,r)}} \right\}.
\label{1def:local_Y}
\end{align}

We also define the global norms $X_T$ and $Y_T$:
\begin{align}
    \begin{aligned}
        \| f \|_{X_T} &:=\sup_{x \in M, 0 < r^2 < T} \| f \|_{X_{B(x,r)}},\\
        \| f_0 \|_{Y^0_T} &:=\sup_{x \in M, 0 < r^2 < T} \| f_0 \|_{Y^0_{B(x,r)}},\\
        \| f_1 \|_{Y^1_T} &:=\sup_{x \in M, 0 < r^2 < T} \| f_1 \|_{Y^1_{B(x,r)}},\\
        \| f \|_{Y_T} &:=\inf_{f = f_0 + \nabla^\ast f_1} \left\{ \| f_0 \|_{Y^0_T} + \| f_1 \|_{Y^1_T} \right\}.
    \end{aligned}
\label{eqn:X_and_Y_norm}    
\end{align}

Compared with Koch-Lamm~\cite{koch2012geometric}, the only differences are the extra volume terms $|B(x,r)|$.  Therefore, the above definition can be regarded as the renormalized version of the norm defined in~\cite{koch2012geometric}. 

It is clear that $\| \cdot \|_{X_T}$ and $\|\|_{Y_T}$ are really norms, i.e., they commute with scalings, are positive away from the origin and satisfy triangle inequalities.  For example, let us check that $\|\cdot\|_{X_T}$ satisfies the triangle inequality: 
\begin{align*}
    \| u + v \|_{X_T} \leq \| u \|_{X_T} + \| v \|_{X_T}.
\end{align*}
Suppose that the maximum value is achieved at $(x,r)$. Then we have
\begin{align*}
    \begin{aligned}
        \| u + v \|_{L^\infty(P(x,r^2))} \leq \| u \|_{L^\infty(P(x,r^2))} + \| v \|_{L^\infty(P(x,r^2))}.
    \end{aligned}
\end{align*}
Similar results hold for the $L^2$ and $L^{n+4}$ norms. Thus, $\| \cdot \|_{X_T}$ satisfies the triangle inequality.

The Banach spaces $X_T$ and $Y_T$ are defined as the completions of smooth tensor fields on $M$ with respect to the corresponding norms.  We shall apply Banach's fixed point principle for contraction mappings in $X_T$.  Note that $X_T$ is a complete functional space contained in $L^\infty(M)$. \\

Suppose $g_t$ satisfies the Ricci-Deturck flow with the background metric $\bar{g}=g_0$.
Define
\begin{align}
    h:=h_t=g_t-g_0.    \label{eqn:SJ15_10}
\end{align}
It follows from Proposition \ref{prop:RD_operator_compute} that
\begin{align*}
    \partial_t h_t=\partial_t g_t=P_{\bar{g}}(g_t)=P_{\bar{g}}(g_0+h_t)=P_{\bar{g}}(g_0) + L h_t +\cQ [h_t]. 
\end{align*}
Denote $h_t$ by $h$ for simplicity. The above equation is equivalent to the following.
\begin{align}
    (\partial_t-L) h=P_{\bar{g}}(g_0) + \cQ [h].  \label{eqn:SI_4}
\end{align}
Since $\bar{\gamma}_{ij}^k=\bar{\Gamma}_{ij}^k-\Gamma_{ij}^k=0$ now, it is clear that
 \begin{align*}
        L h _{ij} 
        =\Delta h_{ij} + (R_{ipqj} + R_{jpqi} )h_{pq}  -  h_{ip}R_{pj}  - h_{jp}R_{pi} =\Delta_L h_{ij}
 \end{align*}
 where $\Delta_L$ is the Lichnerowitz Laplacian with respect to the metric $g_0$. 
 For simplicity, we can write
 \begin{align}
     \Delta_L h= \Delta h + E*h   \label{eqn:SI1_5}
 \end{align}
 with $|E| \leq C(n)$ by the assumption of bounded curvature (\ref{eqn:SI18_1}). 
 Also, $P_{\bar{g}}g_0=-2Rc(g_0)$, which we denote as $Z$ for simplicity.  Then (\ref{eqn:SI_4}) becomes
\begin{align}
   (\partial_t-L) h =(\partial_t-\Delta_L) h =-2Rc(g_0) + \cQ [h] =Z +\cQ [h].    \label{1RD_ptb_from_g0}
\end{align}
It is clear that $|Z|\leq 2(n-1)$.  
The term $\cQ [h]$ is non-linear and satisfies
\begin{align}
            \cQ [h]=\nabla^\ast \cS [h] + \cR [h].   
        \label{eqn:SA09_02}
\end{align}
If $|h|<\frac12$, it follows from (\ref{eqn:SI18_1}) and direct calculation that 
 \begin{align}
            |\cS [h]| \leq  C|h||\nabla h|, \quad
            |\cR [h]| \leq  C(|h|^2 + |\nabla h|^2).
        \label{eqn:SA09_2}   
 \end{align}

 Let $K$ be the fundamental solution for $\partial _t - L $ in \eqref{1RD_ptb_from_g0}. 
 Then each $h$ solves the differential equation~\eqref{1RD_ptb_from_g0} if and only if it solves the following integral equation: 
 \begin{align}\label{1equ:int_ptb}
    h(x,t) = \int_0^t \int_{M} \big\{ K(x,t;y,s) (Z + \cR [h])(y,s) + \nabla_y K(x,t;y,s) \cS [h](y,s) \big\} \, \dd y \, \dd s, 
 \end{align}
 where $\dd y$ is the canonical measure determined by the Riemannian metric $g_0$.
 Note that $dy$ is independent of time.   We shall regard each solution of (\ref{1equ:int_ptb}) as a fixed point of some contraction mapping. Then the solution can be found via Banach's fixed point theorem. 
 In order to establish the contraction property, we need the following two estimates: Proposition~\ref{1clm:X<Y} and~\ref{1clm:Y<X^2}. 

\begin{prop}\label{1clm:X<Y}
	Suppose $0<T<1$ and $\cQ  \in Y_T $. Then there exists a constant $C = C(n)$ such that 
    \begin{align}
    \left \|  \int_0^t \int_M  K(x,t;y,s)\cQ (y,s)\dd y\dd s \right\|_{X_T} \leq C \|\cQ \|_{Y_T}.
    \label{eqn:SB13_1}
    \end{align}
\end{prop}

\begin{proof}
 Note that $\cQ  =\nabla^\ast  \cS+   \cR$ for some $\cR \in Y^0_T $ and $\cS \in Y^1_T$.  However, this decomposition does not uniquely determine $\cS$ and $\cR$. This gives us the freedom to choose $\cR$ and $\cS$ properly such that
 \begin{align}
     \| \cQ  \|_{Y_T} \leq \| \cR\| _{Y^0_T} + \| \cS\|_{Y^1_T} \leq 2 \| \cQ  \|_{Y_T}.
 \label{eqn:SC01_10}
 \end{align}
 In order to prove (\ref{eqn:SB13_1}), it suffices to show that
    \begin{align}
        \left   \|  \int_{M\times [0,r^2]} \big( K(\cdot ,\cdot ;y,s)  \cR(y,s) +   
        \nabla_y  K(\cdot,\cdot;y,s)   \cS(y,s) \big)  \dd y \dd s \right  \|_{X_{B(x_0 ,r)}} \leq C  \| \cQ   \|_{Y_T}
    \label{eqn:SA11_2}    
    \end{align}
 for all $x_0 \in M, r \in (0, \sqrt{T}]$. 

  Define
    \begin{align}
    \xi(x,t) &:=    
        \int_{M\times [0,t]} \big\{ K(x,t;y,s) \cR(y,s) +\nabla_y K(x,t;y,s)   \cS(y,s) \big\}  \dd y \dd s. 
        \label{1eqn:h0add}
    \end{align}
    It follows from this and from the definition of the local $X$-norm in (\ref{1def:local_X}) that the left hand side of (\ref{eqn:SA11_2}) is the same as the one in (\ref{eqn:SA11_2}). 
    \begin{align}
        \| \xi \|_{X_{B(x_0,r)}} &:=\| \xi \|_{L^\infty(P(x_0,r^2))}
         + \frac{\| \nabla \xi \|_{L^2(P(x_0,r^2))}}{|B(x_0,r)|^{\frac{1}{2}}}
         +r^{\frac{n+2}{n+4}} 
         \frac{\| \nabla \xi \|_{L^{n+4}(\Omega(x_0,r^2))}}{|B(x_0,r)|^{\frac{1}{n+4}}}, 
    \label{eqn:SC01_6}     
    \end{align}
    which consists of three parts.  We shall estimate them term by term.\\
	
	\textit{Step 1. Estimate of the $L^\infty$ term.}\\
	
	Fix any $(x,t) \in P(x_0,r)=B(x_0, r)\times (0,r^2)$. 
	It is clear that
	\begin{align}
    \begin{aligned}
			|\xi(x,t)|&\leq \underbrace{ \int_{ \Omega(x,t)} | K(x,t;y,s)  \cR(y,s)| +|   \nabla_y  K(x,t;y,s)  \cS(y,s) |  \dd y\dd s}_{I}\\
			&\quad +  \underbrace{\int_{M\times [0,t] \backslash \Omega(x,t)} | K(x,t;y,s)\cR(y,s)| 
            +|\nabla_y K(x,t;y,s) \cS(y,s)|  \dd y\dd s}_{II},
    \end{aligned}
    \label{eqn:SI02_2}
	\end{align}
    where $\Omega(x,t)=B(x,\sqrt{t}) \times (\frac{t}{2},t)$. We shall estimate I and II separately using the following heat kernel estimate (cf. Proposition~\ref{prop:SI06_0}) repeatedly. 
    \begin{align}
       \sqrt{t-s}|\nabla_y K(x,t;y,s)|+|K(x,t;y,s)| \leq  
       \frac{C}{|B(x,\sqrt{t-s})|} e^{-\frac{d^2(x,y)}{4D(t-s)}}.
       \label{eqn:SI02_3}
    \end{align}

	We first estimate part I in (\ref{eqn:SI02_2}). 
    By H\"older inequality, it is clear that 
	\begin{align}
			I \leq  \| K(x,t;\cdot, \cdot )\|_{L^{\frac{n+4}{n+2}}(\Omega(x,t))}  
            \|\cR\|_{L^{\frac{n+4}{2}}(\Omega(x,t))}
			  +  \| \nabla_y K (x,t;\cdot, \cdot) \|_{L^{\frac{n+4}{n+3}}(\Omega(x,t))}  
            \|\cS\|_{L^{n+4}(\Omega(x,t))}.
            \label{eqn:SI02_5}
	\end{align}
	In light of (\ref{eqn:SI02_3}), we have
    \begin{align*}
        |K(x,t;y,s)|\leq C | B(x,\sqrt {t-s})|^{-1}, \quad \int_M |K(x,t;y,s)| dy \leq C. 
    \end{align*}
	Regarding $|K(x,t;y,s)|^{\frac{n+4}{n+2}}$ as $|K(x,t;y,s)| \cdot |K(x,t;y,s)|^{\frac{2}{n+2}}$ and setting $\tau=t-s$ implies that
	\begin{align}
			&\quad \int_{t/2}^{t} \int_{B(x,\sqrt t )} |K(x,t;y,s)|^{\frac{n+4}{n+2}}\dd y\dd s
             \notag\\
			&\leq C\int_{0}^{t/2} |B(x,\sqrt{\tau})|^{-\frac{2}{n+2}} \left\{ \int_{B(x,\sqrt t )} |K(x,t;y,s)| \dd y \right\} \dd \tau \notag\\
            &\leq C\int_{0}^{t/2} |B(x,\sqrt{\tau})|^{-\frac{2}{n+2}}  \left\{\int_M |K(x,t;y,s)| \dd y \right\}\dd \tau \notag\\
			&\leq  C \int_{0}^{t/2} |B(x,\sqrt{\tau})|^{-\frac{n+4}{n+2}} \dd \tau.   \label{eqn:SA10_6}
	\end{align}
    Note that we have the following volume comparison. 
	\begin{align*}
		\frac{|B(x,\sqrt t)|}{|B(x,\sqrt \tau )|} \leq \frac{V(\sqrt t)}{V(\sqrt \tau)}, \quad \forall \; \tau < t < r^2. 
	\end{align*}
    Here $V(r)$ is the volume of the ball of radius $r$ in space form of constant curvature $-1$.  It is well known that
    \begin{align}
        V(r)=C_n \int_0^r \sinh^{n-1} \rho d\rho \leq C e^{(n-1)r}.   \label{eqn:SI07_10}
    \end{align}
    Since $0<r<\sqrt{T} \leq 1$, we know that $\sinh t$ is comparable to $t$.
    Consequently, 
    \begin{align*}
        \frac{1}{C(n)} \left(\frac{t}{\tau}\right)^{\frac{n}{2}}
        \leq \frac{V(\sqrt t)}{V(\sqrt \tau)} 
        \leq C(n) \left(\frac{t}{\tau}\right)^{\frac{n}{2}}.
    \end{align*}
    Therefore, we have
	\begin{align}
		|B(x,\sqrt{\tau})|^{-1}
        \leq 
        C \tau^{-\frac{n}{2}} t^{\frac{n}{2}} |B(x,\sqrt{t})|^{-1}. \label{equ:compare_ball}
	\end{align}
	Plugging the above inequality into (\ref{eqn:SA10_6}), we obtain
	\begin{align*}
				\int_{t/2}^{t} \int_{B(x,\sqrt t)} |K(x,t;y,s)|^{\frac{n+4}{n+2}}\dd y\dd s
				\leq  C  |B(x,\sqrt t)|^{-\frac{2}{n+2}} t.
	\end{align*}
	Thus, 
	\begin{align}\label{equ:estm_K_int}
			 \|  K(x,t;\cdot,\cdot) \|_{L^{\frac{n+4}{n+2}}(\Omega(x,t))} \leq C {\sqrt t}^{\frac{2n+4}{n+4}} |B(x,\sqrt t)|^{-\frac{2}{n+4}}.
	\end{align}
	
	Similarly, by (\ref{eqn:SI02_3}) we have 
    \begin{align*}
        |\nabla_y K(x,t;y,s)|\leq C(t-s)^{-\frac 1 2} | B(x,\sqrt {t-s})|^{-1}, 
        \quad \int_M |\nabla_y K(x,t;y,s)| dy \leq C(t-s)^{-\frac12}. 
    \end{align*}
	It follows that
	\begin{align*}
			&\quad \int_{t/2}^{t} \int_{B(x,\sqrt t )} |  \nabla_y K(x,t;y,s)|^{\frac{n+4}{n+3}}\dd y\dd s\\
			&\leq  C\int_{0}^{t/2}\tau^{-\frac{1}{2(n+3)}} |B(x,\sqrt{\tau})|^{-\frac{1}{n+3}}  
             \left\{\int_M |\nabla_y K(x,t;y,s)| \dd y \right\} \dd \tau\\
            &\leq  C \int_0^{t/2} \tau^{-\frac{n+4}{2(n+3)} } |B(x,\sqrt \tau)|^{-\frac{1}{n+3}} \dd \tau. 
	\end{align*}
	Using \eqref{equ:compare_ball} again, we have
	\begin{align*}
			&\quad \int_{t/2}^{t} \int_{B(x,\sqrt t )} |  \nabla_y K(x,t;y,s)|^{\frac{n+4}{n+3}}\dd y\dd s\\
			&\leq  C |B(x,\sqrt t)|^{-\frac{1}{n+3}} t^{\frac{n}{2(n+3)}} \int_0 ^{t/2} \tau ^{-\frac{n+2}{n+3}} \dd \tau 
            \leq C t ^{\frac{n+2}{2(n+3)}} |B(x, t)|^{-\frac{1}{n+3}},
	\end{align*}
	which yields
	\begin{align}\label{equ:estm_grad_K_int}
		 \|\nabla_y K (x,t;\cdot, \cdot)\|_{L^{\frac{n+4}{n+3}}(\Omega(x,t))}  \leq C {\sqrt t}^{\frac{n+2}{n+4}}| B(x,\sqrt t)|^{-\frac{1}{n+4}}.
	\end{align}
	Combining formulas (\ref{eqn:SI02_5}),~\eqref{equ:estm_K_int},~\eqref{equ:estm_grad_K_int} and the definition of the Y-norm \eqref{1def:local_Y}, we obtain
	\begin{align}
             I
             &\leq   C  {\sqrt t}^{\frac{2n+4}{n+4}} |B(x,\sqrt t)|^{-\frac{2}{n+4}} \|  \cR  \|_{L^{\frac{n+4}{2}}(\Omega(x,t))} + C{\sqrt t}^{\frac{n+2}{n+4}}| B(x,\sqrt t)|^{-\frac{1}{n+4}}\| \cS \|_{L^{n+4}(\Omega(x,t))}\notag\\
              &\leq  C \| \cR \|_{Y^0_{B(x,\sqrt t)}} + C \| \cS \|_{Y^1_{B(x,\sqrt t)}} 
               \leq  C \| \cQ  \|_{Y_T},
              \label{eqn:SA12_3}
	\end{align}
    which finishes the estimate of part I.  \\

   Then we estimate part II in (\ref{eqn:SI02_2}).  

   Fix an arbitrary point $p \in M$. 
   It follows from the heat kernel estimate (\ref{eqn:SI02_3}) and the volume comparison that 
   \begin{align*}
       &\quad\int_{(B(p,\sqrt{t}) \times [0,t]) \backslash \Omega(x,t)}  \left\{|K(x,t;y,s) \cR(y,s)| +|\nabla_y K(x,t;y,s) \cS(y,s)| \right\} \dd y \dd s\\
       &\leq C \int_{(B(p,\sqrt{t}) \times [0,t]) \backslash \Omega(x,t)}  |B(x,\sqrt{t-s})|^{-1}
        e^{-\frac{d^2(x,y)}{4D(t-s)}} 
        \left\{|\cR|(y,s) 
       + (t-s)^{-\frac 1 2 }|\cS|(y,s) \right\} \dd y \dd s\\
       &\leq \frac{C}{|B(x,\sqrt{t})|} \int_{(B(p,\sqrt{t}) \times [0,t]) \backslash \Omega(x,t)}  \left(\frac{t}{t-s} \right)^{\frac{n}{2}}
        e^{-\frac{d^2(x,y)}{4D(t-s)}} 
        \left\{|\cR|(y,s) 
       + (t-s)^{-\frac 1 2 }|\cS|(y,s) \right\} \dd y \dd s.
   \end{align*}
   For each $(y,s) \in \{B(p, \sqrt{t}) \times [0, t]\} \backslash \Omega(x, t)$, either $t-s \geq \frac{t}{2}$, or $\frac{d^2(x,y)}{t}>1$. No matter what case happens, we have
   \begin{align*}
       \left(\frac{t}{t-s} \right)^{\frac{n}{2}}
        e^{-\frac{d^2(x,y)}{4D(t-s)}} 
        = \left(\frac{t}{t-s} \right)^{\frac{n}{2}}
        e^{-\frac{d^2(x,y)}{4Dt} \cdot \frac{t}{t-s}} 
        \leq C e^{-\frac{d^2(x,y)}{8D t}}. 
   \end{align*}
   Applying the triangle inequality, we have
   \begin{align*}
       d^2(x,y) &\geq |d(x,p)-d(y,p)|^2=d^2(x,p)+d^2(y,p)-2d(x,p)d(y,p)
       \geq d^2(x,p)-2\sqrt{t}d(x,p)\\
       &\geq \frac12 d^2(x,p)-2t.
   \end{align*}
   Combining the previous two inequalities yields
   \begin{align*}
       \left(\frac{t}{t-s} \right)^{\frac{n}{2}}
        e^{-\frac{d^2(x,y)}{4D(t-s)}} 
        \leq C e^{-\frac{d^2(x,p)}{16D t}}. 
   \end{align*}
   Plugging this into the integral estimate above, we have
    \begin{align*}
       &\quad\int_{(B(p,\sqrt{t}) \times [0,t]) \backslash \Omega(x,t)}  \left\{|K(x,t;y,s) \cR(y,s)| +|\nabla_y K(x,t;y,s) \cS(y,s)| \right\} \dd y \dd s\\
       &\leq C 
       e^{-\frac{d^2(x,p)}{16 D t}}
        \left( \frac{\|\cR\|_{L^1(P(p, t ))}}{|B(x,\sqrt{t})|} 
        +\frac{\|\cS\|_{L^2(P(p, t))}}{|B(x,\sqrt{t})|^{\frac12}}）\right)
        \leq C e^{-\frac{d^2(x,p)}{16 D t}} \|\cQ \|_{Y_T},
   \end{align*}
   where we applied the H\"older inequality. 
   Now we choose $p_i$ so that $\{B(p_i, \sqrt t)\}_i$ is a cover of $M$ and $\{B(p_i, \frac12 \sqrt t)\}_i$ are disjoint. Furthermore, we can choose $p_0=x$.   It follows that 
	\begin{align}
			II 
            \leq  C \sum_{i} \exp\left(-\frac{d^2(x,p_i)}{ 16 Dt}\right)  \|\cQ \|_{Y_T}.
            \label{1estm:II_bar}
	\end{align}
	Define a subset 
	$$A_m := \{y\in M : (m-1)\sqrt t \leq d(x,y)< m \sqrt t\} .$$
	Let $N_m$ be the number of $i $'s such that $p_i \in A_m $. Since $\{B(p_i, \frac{1}{2}\sqrt t)\} $ are disjoint, we have the following volume inequality:
    \begin{align*}
       N_m \times   \min_{i:p_i\in A_m } |B(p_i, \frac 1 2 \sqrt t)| 
       \leq |A_m|.  
    \end{align*}
    Volume comparison implies that $N_m \leq \exp(C m)$ for some constant $C=C(n)$. Then 
	\begin{align}
			\sum_{i} \exp \left( -\frac{d^2(x,p_i)}{16 Dt} \right)  
            \leq \sum_m\exp  \left( -\frac{m^2}{16 D} \right) N_m
			\leq C \sum_m  \exp \left ( -C' m^2 \right )
			\leq C.
    \label{eqn:SA12_4}        
	\end{align}
	Combining (\ref{1estm:II_bar}) and (\ref{eqn:SA12_4}), we obtain the estimate of part II. 
        \begin{align}
            II \leq C  \|\cQ \|_{Y_T}.    \label{eqn:SI02_4}
        \end{align}
        
    Putting (\ref{eqn:SA12_3}) and (\ref{eqn:SI02_4}) together, we arrive at
    \begin{align}
        \| \xi \|_{L^\infty(P(x_0,r^2))} \leq C \|\cQ \|_{Y_T}. 
        \label{equ:|h|<Q_Y}
    \end{align}
    This finishes the proof of Step 1. \\
	
	\textit{Step 2. Estimate of the $W^{1,2}$ term.}\\
	
	We choose a cut-off function $\eta$ such that $\eta = 1 $ in $B(x_0,r) $ and $\eta =0 $ out of $B(x_0,2r) $ with gradient bounded by $C/r $. Since $\xi$ is given by the integral (\ref{1eqn:h0add}), $\xi $ solves the differential equation $\partial_t \xi-   L \xi =   \cR +   \nabla ^ \ast \cS $ in the weak sense with the initial value $\xi_0=0$. 
    Since there is no ambiguity here, we omit the volume elements.  Integration by parts yields
	\begin{align}
			0\leq {}& \frac{1}{2} \int_M \eta ^2 |\xi(\cdot, r^2)|^2 -\frac{1}{2} \int_M \eta ^2 |\xi(\cdot, 0)|^2 \notag\\
			= {}& \frac{1}{2} \int_0 ^ {r^2} \partial_t \int_M \eta^2 |\xi|^2
			= \int_0 ^{r^2} \int_M \eta^2 \langle   L \xi  +   \cR +   \nabla^*   \cS , \xi \rangle\notag\\
			\leq {}& -\int_0^{r^2} \int_M \eta ^2 |  \nabla \xi | ^2  +\int_0^{r^2} \int_M 2 |  \nabla \xi | |\xi| |  \nabla \eta | | \eta |   + \int_0^{r^2} \int_M \eta ^2 |  E | |\xi|^2 \notag\\
			{}&  + \int_0^{r^2} \int_M \eta ^2 |  \cR| | \xi| + \int_0^{r^2} \int_M \eta ^2 |\cS| |\nabla \xi|  + \int_0^{r^2} \int_M 2 \eta |\nabla \eta| |\cS| |\xi|.
            \label{1estm:grad_L2_int_by_part}
	\end{align}
    Using the elementary Cauchy-Schwarz inequality, we obtain 
    \begin{align*}
		\int_0^{r^2} \int_M \eta^2 |\nabla \xi|^2 \leq {}& C \int_0^{r^2} \int_M  \left\{ |\nabla \eta|^2 |\xi|^2 + \eta^2|E| |\xi|^2 
        +\eta^2|\cR||\xi| +\eta^2 |\cS|^2 \right\}. 
	\end{align*}
    Since $|E| \leq C$, it follows from the choice of $\eta$ that 
	\begin{align}
		\int_0^{r^2} \int_{B(x_0,r)}|  \nabla \xi |^2 
        &\leq C \int_0^{r^2} \int_{B(x_0,2r)} 
        \left\{ (r^{-2}+1) |\xi|^2+|\cR| |\xi|+ |\cS|^2 \right\}. 
    \label{eqn:SI01_12}    
	\end{align}
    By volume comparison, we can choose up to $N=N(n)$ points $x_i$ such that $\{B(x_i, 0.5r)\}$ are disjoint and $B(x_0, 2r) \subset \cup_{i=1}^N B(x_i, r)$.  
    Then we have 
    \begin{align}
		\int_0^{r^2} \int_{B(x_0,r)}|  \nabla \xi |^2 
        &\leq C \sum_{i=1}^N \int_0^{r^2} \int_{B(x_i, r)} 
        \left\{ (r^{-2}+1) |\xi|^2+|\cR| |\xi|+ |\cS|^2 \right\}. 
    \label{eqn:SC01_11}    
	\end{align}
    In light of the pointwise estimate~\eqref{equ:|h|<Q_Y}, we have the following.
    \begin{align}
        &\quad \int_0^{r^2} \int_{B(x_i, r)} 
        \left\{ (r^{-2}+1) |\xi|^2+|\cR| |\xi|+ |\cS|^2 \right\} \notag\\
        &\leq C (1+r^{2})|B(x_i,r)| \|\cQ \|_{Y_T}^2 + C \|\cQ \|_{Y_T} \|\cR\|_{L^1(P(x_i,r^2))} + C \|\cS\|_{L^2(P(x_i,r^2))}^2 \notag\\
        &\leq |B(x_i, r)| 
        \left\{ C(1+r^{2}) \|\cQ \|_{Y_T}^2 +  C \|\cQ \|_{Y_T} \frac{\|\cR\|_{L^1(P(x_i,r^2))}}{|B(x_i,r)|} + C \left(\frac{\|\cS\|_{L^2(P(x_i,r^2))}}{|B(x_i, r)|^{\frac12}} \right)^2 \right\}. 
        \label{eqn:SC01_12}
    \end{align}
    According to the choice of $\cR$ and $\cS$, the estimate (\ref{eqn:SC01_10}) then implies that 
    \begin{align*}
        &\frac{\|\cR\|_{L^1(P(x_i,r^2))}}{|B(x_i,r)|} \leq \| \cR\|_{Y_{B(x_i,r)}^0} 
        \leq \| \cR\|_{Y_T^0} \leq 2\|\cQ \|_{Y_T},  \\
        &\frac{\|\cS\|_{L^2(P(x_i,r^2))}}{|B(x_i,r)|^{\frac12}} \leq \| S\|_{Y_{B(x_i,r)}^1} 
        \leq \| \cS\|_{Y_T^1} \leq 2\|\cQ \|_{Y_T}. 
    \end{align*}
    Plugging them into (\ref{eqn:SC01_12}) and then into (\ref{eqn:SC01_11}), we obtain the following. 
	\begin{align*}
			\frac{\int_0^{r^2} \int_{B(x_0,r)} |\nabla \xi |^2}{|B(x_0,r)|}  
            \leq \sum_{i=1}^N  C \frac{|B(x_i, r)|}{|B(x_0, r)|} 
            \left\{ (1+r^2) \|\cQ \|_{Y_T}^2\right\} 
            \leq C \|\cQ \|_{Y_T}^2, 
	\end{align*}
    where we used the volume comparison and the fact $r<\sqrt{T}<1$ in the last step. 
    It follows that
    \begin{align}
       \frac{\| \nabla \xi \|_{L^2(P(x_0,r^2))}}{|B(x_0,r)|^{\frac{1}{2}}} 
       \leq C  \|\cQ \|_{Y_T}.   \label{eqn:SC01_7}
    \end{align}
    This finishes the proof of Step 2. \\
	
	\textit{Step 3. Estimate of the $W^{1, n+4}$ term.}\\
	
	Suppose $(x,t)\in \Omega(x_0,r^2)$. Notice that
	\begin{align*}
			\nabla \xi (x,t) 
            =\nabla_x \xi (x,t)
            =\int_{M \times [0,t]}  \nabla_x K(x,t;y,s)\cR(y,s) 
             +\nabla_x \nabla_y K(x,t;y,s)\cS(y,s) \dd y \dd s. 
	\end{align*} 
    Therefore, we have 
	\begin{align}
    \begin{aligned}
       |\nabla \xi |(x,t) 
            &\leq  \underbrace{\int_{ \{M \times [0,t]\} \cap \Omega(x_0,r^2)}   |  \nabla_x K(x,t;y,s)  \cR(y,s)  +   \nabla_x \nabla_y K(x,t;y,s)  \cS(y,s)| \dd y \dd s}_{Inn(x,t)}\\
			&\quad + \underbrace{\int_{\{M \times [0,t]\} \backslash \Omega(x_0,r^2) } |  \nabla_x K(x,t;y,s)  \cR(y,s)  +   \nabla_x \nabla_y K(x,t;y,s)  \cS(y,s)| \dd y \dd s}_{Out(x,t)}. 
    \end{aligned}
	\label{eqn:SI02_6}		
	\end{align} 
    The estimate of $Out(x,t)$ is almost the same as that of part II in (\ref{eqn:SI02_2}).
    The only difference is that the convolution kernel now is $\nabla_x K(x,t;y,s)$, instead of $K(x,t;y,s)$.  However, in the current setting we still have a Gaussian-type estimate. 
    \begin{align*}
        |\nabla_x K(x,t;y,s)| + \sqrt{t-s} |\nabla_x \nabla_y K(x,t;y,s)| \leq 
         \frac{1}{\sqrt{t-s}} \cdot \frac{C}{|B(x,\sqrt{t-s})|} \cdot e^{-\frac{d^2(x,y)}{4D(t-s)}}.
    \end{align*}
    Therefore, the same argument for~\eqref{1estm:II_bar} applies and yields 
    \begin{align*}
        Out (x,t) < C t^{-\frac12}  \|\cQ \|_{Y_T}.
    \end{align*}
    Consequently, we have
	\begin{align}\label{1estm:Out bar}
			r^{\frac{n+2}{n+4}}|B(x_0,r) |^{-\frac{1}{n+4}} 
            \|Out\|_{L^{n+4}(\Omega(x_0,r^2))} \leq C   \|\cQ \|_{Y_T}.
	\end{align} 
	Therefore, we only need to estimate the $L^{n+4}(\Omega(x_0,r^2)) $ norm of $Inn$. Without loss of generality, we assume that $\cR$ and $\cS$ are supported in $\Omega(x_0, r^2) $. 
    For any  $(x,t) \in \Omega (x_0,r^2)$,  we define 
    \begin{align*}
       &I_R (x,t):= \int_{\frac{r^2}{2}} ^{t} \int_{M}   \nabla_x  K (x,t;y,s)  \cR(y,s)\dd y \dd s, \\
       &I_S (x,t):= \int_{\frac{r^2}{2}}  ^{t} \int_{M}   \nabla_x \nabla_y K (x,t;y,s)  \cS(y,s)\dd y \dd s.
    \end{align*}
    Then it is clear that $Inn (x,t) = I_R(x,t) +  I_S (x,t)$. 
	We shall estimate $I_R$ and $I_S$ separately. 
    
	We first estimate $I_R$.
    Since $\frac{1}{n+4} + 1 = \frac{n+3}{n+4} + \frac{2}{n+4} $, we can apply Young's inequality to obtain
	\begin{align}
			 \| I_R \|_{L^{n+4}(\Omega(x_0,r^2))}  \leq {}& \sup_{(x,t)\in \Omega(x,r^2)} \|    \nabla_x K(x,t;\cdot, \cdot )  \|_{L^{\frac{n+4}{n+3}}(\Omega(x_0,r^2))}  \|   \cR  \|_{L ^{\frac{n+4}{2}}(\Omega(x_0,r^2))}. 
    \label{eqn:SA09_5}
	\end{align}
    Applying the heat kernel estimate, we obtain the following integral estimate similar to (\ref{equ:estm_grad_K_int}).
   \begin{align*}
		 \|\nabla_x K (x,t;\cdot, \cdot)\|_{L^{\frac{n+4}{n+3}}(\Omega(x,t))}  \leq C {\sqrt t}^{\frac{n+2}{n+4}}| B(x,\sqrt t)|^{-\frac{1}{n+4}} \leq C  {r}^{\frac{n+2}{n+4}}| B(x_0,r)|^{-\frac{1}{n+4}}, 
	\end{align*}
    where we used volume comparison in the last inequality. 
    Putting this into (\ref{eqn:SA09_5}), we obtain 
    \begin{align}
        &\quad r^{\frac{n+2}{n+4}}|B(x_0,r) |^{-\frac{1}{n+4}}	 \|  I_R \|_{L^{n+4}(\Omega(x_0,r^2))} \notag\\
        &\leq C  r^{\frac{n+2}{n+4}}|B(x_0,r) |^{-\frac{1}{n+4}}	\cdot  {r}^{\frac{n+2}{n+4}}| B(x_0,r)|^{-\frac{1}{n+4}} 
        \cdot   \|   \cR  \|_{L ^{\frac{n+4}{2}}(\Omega(x_0,r^2))} 
        \leq C  \|\cR\|_{Y_{B(x_0,r)}^0}. 
        \label{1estm:I_R bar}
	\end{align}

    We move on to estimate $I_S$. 
    This term is more delicate, as it deals with the convolution with $\nabla_x \nabla_y K(x,t;y,s)$, whose estimate is weaker. We define  
	\begin{align}
	    \xi_S(x,t):= \int_{\frac{r^2}{2}}^{t} \int_{M}   \nabla_y K (x,t;y,s)\cS(y,s)\dd y \dd s.   \label{eqn:SC07_1}
	\end{align}
    Then $\xi_S$ solves the equation
    $(\partial_t-L) \xi_S =\nabla^\ast  \cS$ in $\Omega(x_0,r^2)$ 
    with initial value $\xi_S|_{t=r^2/2}=0$. 
    Furthermore, we have
    \begin{align}
        \nabla \xi_S (x,t) = I_S(x,t) =\int_{\frac{r^2}{2}}^{t} \int_{M}  \nabla_x \nabla_y K (x,t;y,s)\cS(y,s)\dd y \dd s.    \label{eqn:SC01_9}
    \end{align}

    Recall that $\xi$ satisfies $(\partial_t-L) \xi=\cR + \nabla^* \cS$. We have derived estimates in ~\eqref{1estm:grad_L2_int_by_part}, (\ref{eqn:SC01_11}), and (\ref{eqn:SC01_12}). Now we apply the same method on $\xi_S$, which satisfies $(\partial_t-L) \xi_S= \nabla^* \cS$. Similarly to (\ref{eqn:SC01_11}), we have the following.
	\begin{align}
			 \int_{\frac{r^2}{2}}^{r^2}\int_{B(x_0,r)}   |\nabla \xi_S|^2 
			 \leq C \int_{\frac{r^2}{2}}^{r^2}\int_{B(x_0,r)}( |  \cS|^2 + r^{-2}|\xi_S|^2 ). 
    \label{eqn:SC07_4}         
	\end{align}
    In order to estimate the integral of $|\xi_S|^2$, we note that the map
    \begin{align}
    \mathcal{T}: f \mapsto \mathcal{T}(f)=\int_{\frac{r^2}{2}}^{t} \int_{M}   |\nabla_y K|(x,t;y,s) f(y,s)\dd y \dd s
    \label{eqn:SC07_5}
    \end{align}
    is an operator that transforms functions defined on $\Omega(x_0, r^2)$.  
    Recall the heat kernel estimate:
    \begin{align*}
        \int_M |\nabla_y K|(x,t;y,s) \dd y \leq C(t-s)^{-\frac12}, 
        \quad
        \int_M |\nabla_y K|(x,t;y,s) \dd x \leq C(t-s)^{-\frac12}. 
    \end{align*}
    Thus, it follows directly from (\ref{eqn:SC07_5}) that 
    \begin{align}
        &\|\mathcal{T}(f)\|_{L^{\infty}(\Omega(x_0, r^2))} \leq C \cdot  \|f\|_{L^{\infty}(\Omega(x_0, r^2))} \cdot \int_{\frac{r^2}{2}}^{r^2} \frac{1}{\sqrt{t-s}} \dd s \leq C r \|f\|_{L^{\infty}(\Omega(x_0, r^2))}.
    \label{eqn:SC07_2}    
    \end{align}
    On the other hand,  note that 
     \begin{align}
        \|\mathcal{T}(f)\|_{L^{1}(\Omega(x_0, r^2))} 
        &\leq \int_{\frac{r^2}{2}}^{r^2} \int_M  \int_{\frac{r^2}{2}}^{t} \int_{M}   |\nabla_y K| (x,t;y,s) |f|(y,s)\dd y \dd s  \dd x \dd t \notag\\
        &=\int_{\frac{r^2}{2}}^{r^2} \int_M \left( \int_{s}^{r^2} \int_M  |\nabla_y K| (x,t;y,s) dx dt\right) |f|(y,s) \dd y \dd s \notag\\
        &\leq C r \int_{\frac{r^2}{2}}^{r^2} 
        \int_M |f|(y,s) \dd y \dd s
        \leq Cr \|f\|_{L^{1}(\Omega(x_0, r^2))}.  
        \label{eqn:SC07_3}
    \end{align}
    By virtue of the Riesz interpolation (cf. Theorem 0.1.13 of~\cite{sogge2017fourier}),  (\ref{eqn:SC07_2}) and (\ref{eqn:SC07_3}) yield that
    \begin{align*}
        &\|\mathcal{T}(f)\|_{L^2(\Omega(x_0, r^2))} 
        \leq C r \|f\|_{L^2(\Omega(x_0, r^2))}.
    \end{align*} 
    Note that (\ref{eqn:SC07_1}) implies that $|\xi_S| \leq |\mathcal{T}(|\cS|)|$, which in turn implies that 
    \begin{align}
     \|\xi_S\|_{L^2(\Omega(x_0, r^2))} \leq
     \|\mathcal{T}(|\cS|)\|_{L^2(\Omega(x_0, r^2))} 
        \leq C r \|\cS\|_{L^2(\Omega(x_0, r^2))}.  
    \label{eqn:SI04_9}
    \end{align}
    Plugging this inequality into (\ref{eqn:SC07_4}) and applying (\ref{eqn:SC01_9}), we obtain 
    \begin{align*}
     \left  \|  \int_0^t\int_{M}  
     \nabla_x \nabla_y K (x, t;y,s)   \cS(y,s)\dd y \dd s \right \|_{L^2(\Omega(x_0, r^2))}^2
     =\int_{\frac{r^2}{2}}^{r^2}\int_{B(x_0,r)} |  \nabla \xi_S|^2
	 \leq C   \| \cS \|_{L^2(\Omega(x_0, r^2))}^2.   
	\end{align*}
    Therefore, $\mathcal{K} =\nabla_x \nabla_y K$ is a bounded operator 
    from $L^2(\Omega(x_0,r^2), Sym^2 T^*M \otimes T^*M)$ to itself.
    In light of the curvature and curvature derivative estimate (\ref{eqn:SI18_1}), 
    we have the Gaussian-type estimate of $\nabla_x\mathcal{K}$, by Proposition~\ref{prop:SI06_0}. Therefore, we can apply the Calder\'on-Zygmund inequality (cf. Han-Wang~\cite{hanwang_CZ}) in the case $p=n+4$ to obtain the following result. 
	\begin{align*}
	\| I_S \|_{L^{n+4}(\Omega(x_0, r^2))}
    =\left  \|  \int_0^t\int_{M}  
     \nabla_x \nabla_y K (x, t;y,s)   \cS(y,s)\dd y \dd s \right \|_{L^{n+4}(\Omega(x_0, r^2))}
	\leq  C  \|    \cS  \|_{L^{n+4}(\Omega(x_0, r^2))},
	\end{align*}
	which yields
	\begin{align}
	    r^{\frac{n+2}{n+4}}|B(x_0,r) |^{-\frac{1}{n+4}} \|  I_S  \|_{L^{n+4}(\Omega(x_0, r^2))}
	\leq  C  \| \cS\|_{Y_{B(x,r)}^1}.
    \label{eqn:SJ15_1}
	\end{align}  
	Combining the estimate of $Out$ , $I_R$ and $I_S $ in \eqref{1estm:Out bar}, \eqref{1estm:I_R bar} and \eqref{eqn:SJ15_1}, we obtain 
    \begin{align}
        r^{\frac{n+2}{n+4}} 
        \frac{\| \nabla \xi \|_{L^{n+4}(\Omega(x_0,r^2))}}{|B(x_0,r)|^{\frac{1}{n+4}}}
         \leq C  \|\cQ \|_{Y_T}.    \label{eqn:SC01_8}
    \end{align}
    Thus, we arrive at the desired estimate of Step 3.\\

    Plugging (\ref{equ:|h|<Q_Y}), (\ref{eqn:SC01_7}) and (\ref{eqn:SC01_8}) into (\ref{eqn:SC01_6}), we obtain (\ref{eqn:SA11_2}) and consequently finish the proof of the whole proposition.
\end{proof}

\begin{prop}
\label{1clm:Y<X^2}
    For any $T<1 $,  there exist a small constant $\delta<1$
    and a large constant $C=C(n)$ such that for each $x \in M$ and $r \in (0, \sqrt{T})$, the following property holds:
    \begin{itemize}
        \item  If $ \|  h  \|_{X_T} < \delta $, then
          \begin{align}\label{1Y<X}
			 \|  \cQ [h] \|_{Y_{B(x,r)}} \leq  C \|  h \|_{X_{B(x,r)}}^2.
	      \end{align} 
        \item If $ \| v \|_{X_T} < \delta$ and $\|  w \|_{X_T} <\delta$, then
          \begin{align}
			 \|\cQ [v] -\cQ [w]\|_{Y_{B(x,r)}} 
             \leq C (\|v\|_{X_{B(x,r)}}+  \|w\|_{X_{B(x,r)}})\|v-w\|_{X_{B(x,r)}}.  
          \label{eqn:SI04_6}   
	      \end{align}
    \end{itemize}
\end{prop}
\begin{proof}
	
   We first show the inequality~\eqref{1Y<X}.
   
    In light of (\ref{eqn:SB13_1}), 
    it follows from the definition of the norm $X$ and $Y$ (cf. \eqref{1def:local_X} and \eqref{1def:local_Y}) that 
    \begin{align*}
         \|  | \nabla h |^2   \|_{Y^0_{B(x,r)}} 
        &=\frac{\|  |\nabla h|^2  \|_{L^1(P(x,r^2))}}{|B(x,r)|}  + r^{\frac{2n+4}{n+4}} \frac{\| |\nabla h|^2 \|_{L^{\frac{n+4}{2}}(\Omega(x,r^2))}}{|B(x,r)|^{\frac{2}{n+4}}}\\
        &\leq  C   \bigg ( \frac{\| \nabla h \|_{L^2(P(x,r^2))}}{|B(x,r)|^{\frac12}} 
        + r^{\frac{n+2}{n+4}}
        \frac{\|    \nabla h  \|_{L^{n+4}(\Omega(x,r^2))}}{|B(x,r)|^{\frac{1}{n+4}}} \bigg)^2
        \leq C \|h\|_{X_{B(x,r)}}^2.
    \end{align*}
    We also have 
    \begin{align*}
          \|  |h|^2  \|_{Y^0_{B(x,r)}} 
        &\leq   \|  h  \|_{L^\infty(P(x,r^2))}^2 
        \bigg( \frac{\| 1   \|_{L^1(P(x,r^2))}}{|B(x,r)|}
        + r^{\frac{2n+4}{n+4}} 
        \frac{\| 1 \|_{L^{\frac{n+4}{2}}(\Omega(x,r^2))}}{|B(x,r)|^{\frac{2}{n+4}}} \bigg)\\
       &\leq C r^2 \|  h  \|_{L^\infty(P(x,r^2))}^2
        \leq C \|h\|_{X_{B(x,r)}}^2. 
    \end{align*}
    Combining the previous two inequalities, we arrive at
	\begin{align}\label{1bdd:RY<hX^2}
			 \|  \cR [h]  \|_{ Y^0_{B(x,r)} } \leq 
             C\left\{ \|  |  \nabla h |^2  \|_{Y^0_{B(x,r)}} 
             +\|  |h|^2  \|_{Y^0_{B(x,r)}} \right\}
			\leq C  \|  h  \|_{X_{B(x,r)}}^2.
	\end{align}
	Similarly, we have the pointwise estimate $\cS [h] \leq  C|  \nabla h | |h|$ by (\ref{eqn:SA09_2}). Then we have
	\begin{align}
			 \|  \cS [h]  \|_{Y^1_{B(x,r)}} 
			&\leq  C \|  h  \|_{L^\infty(P(x,r^2))}  \|   \nabla h  \|_{Y^1_{B(x,r)}}
            \leq C \|  h  \|_{L^\infty(P(x,r^2))}  \|h\|_{X_{B(x,r)}} 
			\leq C  \|   h  \|_{X_{B(x,r)}}^2. 
             \label{eqn:SA14_1}
	\end{align}
	Since $  \cQ [h] =   \cR [h] +   \nabla ^\ast   \cS [h] $, it is clear that~\eqref{1Y<X} follows directly from the combination of~\eqref{1bdd:RY<hX^2} and~\eqref{eqn:SA14_1}.\\

   Now we prove (\ref{eqn:SI04_6}). 
    
   Recall that (cf. (\ref{eqn:SB24_2}) and (\ref{eqn:SC16_2}))
   \begin{align*}
     \cS [h]_{ij}^k =\big ((g+h)^{kl}-   g^{kl}\big )  \nabla_l h_{ij}.   
   \end{align*}
   It follows from the direct calculation that
	\begin{align*}
			|\cS [v]-\cS [w]| \leq C \left\{ |v-w| \cdot |\nabla v| +|w| \cdot |\nabla (v-w)| \right\}.
	\end{align*}
	By definition, we have
    \begin{align*}
			 \|  \cS [v] - \cS [w] \|_{Y^1_{B(x,r)}} 
			&\leq  C  \frac{\|  |v-w| |\nabla v|  \|_{L^2(P(x,r^2))}}{|B(x,r)|^{\frac 1 2 }} 
			+ C r^{\frac{n+2}{n+4}} \frac{\|  |v-w| |\nabla v |  \|_{L^{n+4}(\Omega(x,r^2))}}{|B(x,r)|^{\frac{1}{n+4}}}\\
			& \quad + C   \frac{\|  |w| |\nabla (v-w)|  \|_{L^2(P(x,r^2))}}{|B(x,r)|^{\frac12}} 
			+ C r^{\frac{n+2}{n+4}} \frac{\|  |w| |\nabla (v-w) |  \|_{L^{n+4}(\Omega(x,r^2))}}{|B(x,r)|^{\frac{1}{n+4}}}\\
			&\leq  C \|  v-w \|_{L^{\infty}(B(x,r))}  \|  v  \|_{X_{B(x,r)}} +   C  \|  w \|_{L^{\infty}(B(x,r))}  \|  v-w  \|_{X_{B(x,r)}}.
	\end{align*}
    Since the $L^{\infty}$-norm is bounded by the $X$-norm, we obtain the following. 
    \begin{align}
        \|\cS [v]-\cS [w]\|_{Y^1_{B(x,r)}}  
        \leq
        C (\|v\|_{X_{B(x,r)}} + \|w\|_{X_{B(x,r)}}) \|v-w\|_{X_{B(x,r)}}.
        \label{eqn:SI04_4}
    \end{align}

	By the explicit formula of $\cR [h]$, we have
	\begin{align}
			&\quad |\cR [v] - \cR [w]| \notag\\
            &\leq  C \big\{ \underbrace{|\nabla (v-w)| (| \nabla v| + |\nabla w|)}_{I} 
            +\underbrace{|v-w| (|v| + |w|)}_{II}
            +\underbrace{|\nabla (v-w)| |v|+|\nabla w | |v-w|}_{III} \big\}. 
            \label{eqn:SI04_0}
	\end{align}
	As before, we shall estimate the above inequality term by term. 

    We first estimate $I$. 
    The H\"older inequality implies
    \begin{align*}
        \frac{\||\nabla (v-w)||\nabla v|\|_{L^1(P(x,r^2))}}{|B(x,r)|}
        \leq \frac{ \|\nabla (v-w) \|_{L^2(P(x,r^2))} \cdot  \|  \nabla v  \|_{L^2(P(x,r^2))} }{|B(x,r)|}
        \leq \|v-w\|_{X_{B(x,r)}}  \|v\|_{X_{B(x,r)}}. 
    \end{align*}
    Similarly, we have
    \begin{align*}
       r^{\frac{2n+4}{n+4}} \cdot \frac{\||\nabla (v-w)||\nabla v|\|_{L^{\frac{n+4}{2}}(\Omega(x,r^2))}}{|B(x,r)|^{\frac{2}{n+4}}}
       &\leq 
       r^{\frac{2n+4}{n+4}} \frac{ \|\nabla (v-w)\|_{L^{n+4} (\Omega(x,r^2))} \cdot \|\nabla v\|_{L^{n+4}(\Omega(x,r^2))} }{|B(x,r)|^{\frac{2}{n+4}}} \\
       &\leq \|v-w\|_{X_{B(x,r)}}  \|v\|_{X_{B(x,r)}}.
    \end{align*}
	Thus, 
	\begin{align}
              \|I\|_{Y^0_{B(x,r)}} 
			 &=\| |\nabla (v-w)| |\nabla v| + |\nabla w| |\nabla (v-w)| \|_{Y^0_{B(x,r)}} \notag\\
             &\leq \| |\nabla (v-w)| |\nabla v| \|_{Y^0_{B(x,r)}} 
             +\|  |\nabla w| |\nabla (v-w)| \|_{Y^0_{B(x,r)}}  \notag\\
			 &\leq  ( \|v\|_{X_{B(x,r)}}+  \|w\|_{X_{B(x,r)}}) \|v-w\|_{X_{B(x,r)}}. 
             \label{eqn:SI04_1}
	\end{align}

    Then we estimate $II$. 
    It is obvious that both $\frac{\||v-w|\cdot |v| \|_{L^1(P(x,r^2))} }{|B(x,r)|}$ and
    $r^{\frac{2n+4}{n+4}} \frac{\|  |v-w| \cdot |v|  \|_{L^{\frac{n+4}{2}}(\Omega(x,r^2))}}{|B(x,r)|^{\frac{2}{n+4}}}$ are dominated by $r^{2} \|v-w\|_{L^\infty(B(x,r))} \cdot   \|v\|_{L^\infty(B(x,r))}$, 
    and consequently by $r^{2} \|v-w\|_{X_{B(x,r)}} \cdot \|v\|_{X_{B(x,r)}}$.
    Since $r^2<T$ is uniformly bounded, we have 
    \begin{align}
		 \|II\|_{Y^0_{B(x,r)}} 
         =\||v-w| |v| + |v-w| |w| \|_{Y^0_{B(x,r)}} 
         \leq C ( \| v \|_{X_{B(x,r)}} + \| w \|_{X_{B(x,r)}}) \|v-w\|_{X_{B(x,r)}}.
    \label{eqn:SI04_2}
	\end{align}
    
    We then move on to estimate the  term $III$.  The H\"older inequality implies 
    \begin{align*}
        \frac{\||\nabla (v-w)||v|\|_{L^1(P(x,r^2))}}{|B(x,r)|}
        \leq \frac{ \|\nabla (v-w) \|_{L^2(P(x,r^2))} }{|B(x,r)|^{\frac12}} \cdot 
        \|v\|_{L^{\infty}(P(x,r))}
        \leq \|v-w\|_{X_{B(x,r)}}  \|v\|_{X_{B(x,r)}}. 
    \end{align*}
    The same argument shows that
     \begin{align*}
        r^{\frac{2n+4}{n+4}} \cdot \frac{\||\nabla (v-w)|\cdot |v|\|_{L^{\frac{n+4}{2}}(\Omega(x,r^2))}}{|B(x,r)|^{\frac{2}{n+4}}}
       &\leq  r^2
       \cdot \frac{\||\nabla (v-w)|\|_{L^{\frac{n+4}{2}}(\Omega(x,r^2))}}{|B(x,r)|^{\frac{1}{n+4}}}
       \cdot \|v\|_{L^{\infty}(\Omega(x,r^2))}\\
       &\leq \|v-w\|_{X_{B(x,r)}}  \|v\|_{X_{B(x,r)}}.
    \end{align*}
    Consequently, 
    \begin{align*}
         \||v| |\nabla(v-w)|\|_{Y^0_{B(x,r)}} 
         \leq C \|v-w\|_{X_{B(x,r)}}  \|v\|_{X_{B(x,r)}}.
    \end{align*}
    It follows that 
    \begin{align}
         \|III\|_{Y^0_{B(x,r)}} 
         &\leq  \||v| |\nabla(v-w)|\|_{Y^0_{B(x,r)}} +\||v-w| |\nabla w|\|_{Y^0_{B(x,r)}} \notag\\
         &\leq C \|v-w\|_{X_{B(x,r)}}  (\|v\|_{X_{B(x,r)}} +\|w\|_{X_{B(x,r)}}).
    \label{eqn:SI04_3}
    \end{align}
    
    Plugging (\ref{eqn:SI04_1}), (\ref{eqn:SI04_2}) and (\ref{eqn:SI04_3}) into (\ref{eqn:SI04_0}), 
    we obtain 
	\begin{align}
			 \|\cR [v]-\cR [w]\|_{Y^0_{B(x,r)}} 
             \leq  
             C ( \|v\|_{X_{B(x,r)}}+ \|w\|_{X_{B(x,r)}}) 
             \|v-w\|_{X_{B(x,r)}}.
    \label{eqn:SI04_7}         
	\end{align}
    Therefore, the estimate (\ref{eqn:SI04_6}) follows from the combination of (\ref{eqn:SI04_4}) and (\ref{eqn:SI04_7}). 
\end{proof}

\begin{prop}
\label{prop:K_conv_Z}
    Let $T<1$. Let $Z$ be a bounded tensor field on $M$ 
    and $K$ be the heat kernel of the operator $\partial_t -L$. 
    Then there exists a constant $C=C(n)$ such that
    \begin{align}
        \left \| \int_0^t \int_{M} K(x,t;y,s) Z(y) \dd y\dd s \right \|_{X_T} \leq C T \| Z \|_{L^\infty(M)}
    \label{eqn:SC01_1}    
    \end{align}
    for $0<t <T$ and $x\in M $. 
\end{prop}

\begin{proof}
    Denote
    \begin{align*}
     H(x,t) :=  \int_0^t \int_{M} K(x,t;y,s) Z (y) \dd y\dd s.
    \end{align*}
	Then (\ref{eqn:SC01_1}) is nothing but   
    \begin{align}
         \left \| H \right \|_{X_T} \leq C T \| Z \|_{L^\infty(M)}.
    \label{eqn:SC01_2}     
    \end{align}
    By the $L^1$-estimate of the heat kernel, we have 
    \begin{align*}
        |H(x,t)| \leq C t \| Z \|_{L^\infty(M)},
    \end{align*}
    which yields
    \begin{align}
        \|H\|_{L^{\infty}(P(x,t))} \leq C t \| Z \|_{L^\infty(M)}
        \leq C T \| Z \|_{L^\infty(M)}. 
        \label{eqn:SC01_3}
    \end{align}
    
    Fix $r \leq \sqrt{t}$.
    Choosing a cutoff function as in~\eqref{1estm:grad_L2_int_by_part},  integration by parts then implies that 
    \begin{align*}
        \int_0^{r^2}\int_{B(x,r)} |\nabla H |^2 &\leq \frac{C}{r^2} \int_0^{r^2}\int_{B(x,2r)} |H|^2 +  C\int_0^{r^2}\int_{B(x,2r)} |Z||H|\\
        &\leq C \| Z \|_{L^\infty(M)}^2  
        \cdot |B(x,2r)| \cdot \left\{ r^{-2}\int_0^{r^2} s^2 ds + \int_0^{r^2} s ds \right\}\\
        &\leq C \| Z \|_{L^\infty(M)}^2  
        \cdot |B(x,2r)| \cdot r^4. 
    \end{align*}
    In light of volume comparison, it follows from the above inequality that
    \begin{align}
        \frac{\|\nabla H\|_{L^2(P(x,r^2))}}{|B(x,r)|^{\frac {1}{2}}} \leq C r^2 \cdot \left(\frac{|B(x,2r)|}{|B(x,r)|} \right)^{\frac12}  \| Z \|_{L^\infty(M)} \leq C t \| Z \|_{L^\infty(M)}
        \leq C T \| Z \|_{L^\infty(M)}. 
    \label{eqn:SC01_4}    
    \end{align}
    Replacing $\cQ $ by $Z$, we can argue as Step 3 in Proposition~\ref{1clm:X<Y} and obtain
    \begin{align}
        r^{\frac{n+2}{n+4}} \frac{\|\nabla H (x,t)\|_{L^{n+4}(P(x,r^2))}}{|B(x,r)|^{\frac{1}{n+4}}} \leq C \| Z \|_{Y^0_T} 
        \leq CT \|Z\|_{L^\infty(M)}.
    \label{eqn:SC01_5}    
    \end{align}
    Then it is clear that (\ref{eqn:SC01_2}) follows from the combination of (\ref{eqn:SC01_3}), (\ref{eqn:SC01_4}), (\ref{eqn:SC01_5}) and the definition of the $X_T$-norm. 
\end{proof}

Then we can prove the short-time existence of the Ricci-DeTurck flow.
Recall from (\ref{eqn:SJ15_10}) that $g_t$ solves the Ricci-Deturck flow if and only if
$h_t=g_t-g_0$ solves the differential equation (\ref{1RD_ptb_from_g0}), or the integral equation (\ref{1equ:int_ptb}). 
\begin{thm}
\label{1thm:int_pert_solu_exist}
    For any metric $g_0$ satisfying (\ref{eqn:SI18_1}), there exist a constant $\delta > 0$ and a unique solution $(h_t)_{t \in [0,T]} \in X_T$ to the equation \eqref{1equ:int_ptb}:
    \begin{align}
        h(x,t) = \int_0^t \int_{M} \big\{ K(x,t; y, s) (Z + \cR [h])(y, s) + \nabla_y K(x,t; y, s) \cS [h](y, s) \big\} \, \dd y \dd s,    \label{eqn:SI05_1}
    \end{align}
    with $h_0 = 0$ such that $\| h \|_{X_T} \leq \delta$.
\end{thm}

\begin{proof}
    We shall use the Banach fixed point principle to derive the existence. Let $\Phi(\cdot)$ be an operator from the space $X_T$ to itself, defined by
    \begin{align}\label{1def:int_oprt}
            \Phi(h) := \int_0^t \int_{M} \big\{ K(x,t; y, s) (Z + \cR [h])(y, s) 
     +\nabla_y K(x,t; y, s) \cS [h](y, s) \big\} \, \dd y\, \dd s.
    \end{align}
    Let $B_{\delta, T}:=\{ h \in X_T :  \|h\|_{X_T} \leq \delta \}$. 
    By choosing $\delta$ and $T$ sufficiently small, it is not hard to see that $\Phi$ is a contraction mapping from $B_{\delta,T}$ to itself.
    
    Firstly, in light of Proposition \ref{prop:K_conv_Z}, by choosing $\delta$ very small and 
    $T=\frac{\delta}{20\| Z \|_{L^\infty(M)}C}$, we have 
    \begin{align*}
        \| \Phi(h) \|_{X_T} \leq C T \| Z \|_{L^\infty(M)} + C ^2 \| h \|_{X_T}^2 \leq \delta,  \quad \forall \; h \in B_{\delta, T}. 
    \end{align*}
    Thus, $\Phi$ is a map from $B_{\delta,T}$ to itself.
    
       Secondly, for each pair of points $v,w \in B_{\delta,T}$, we have
        \begin{align*}
            &\quad \| \Phi(v) - \Phi(w) \|_{X_T} \\
            &=\bigg\| \int_{M \times [0,t]} \big\{ K(x,t; y, s) (\cR [v] - \cR [w])(y, s) 
             + \nabla_y K(x,t; y, s) (\cS [v] - \cS [w])(y, s) \big\} \, \dd y\, \dd s \bigg\|_{X_T}.
        \end{align*}
    Using Proposition \ref{1clm:X<Y},  we can estimate
    \begin{align*}
            \| \Phi(v) - \Phi(w) \|_{X_T} \leq C \| \cR [v] - \cR [w] + \nabla^*\cS [v] - \nabla^*\cS [w] \|_{Y_T} 
            \leq C \| \cQ [v] - \cQ [w] \|_{Y_T}.
    \end{align*}
    By choosing $\delta$ small enough so that $2\delta<1$ and $\delta < \frac{1}{4C^2}$,
    we can apply Proposition \ref{1clm:Y<X^2} to obtain
    \begin{align*}
            \| \Phi(v) - \Phi(w) \|_{X_T} \leq {}& C \| \cQ [v] - \cQ [w] \|_{Y_T} 
            \leq \frac 1 2 \| v - w \|_{X_T}. 
    \end{align*}
    Therefore,  $\Phi$ is a contraction mapping from $B_{\delta,T}$ to itself. Thus, the Banach fixed point theorem implies that there exists a unique solution $h \in X_T$ 
    with $\| h \|_{X_T} \leq \delta$. In other words, $h$ satisfies equation (\ref{eqn:SI05_1}). 
    Since the heat kernel $K$ is smooth, we know that $h$ is smooth and satisfies the initial condition $h_0=0$. 
\end{proof}

\section{Continuous dependence}
\label{sec:cd}

In this section, we study the deformation of a Ricci flow with bounded curvature. 
Let $\{(M, g_t), 0 \leq t \leq 1\}$ be a Ricci flow that satisfies the curvature bound.
\begin{align}
    |Rm|_{g_t} \leq 1.     \label{eqn:SI01_6}
\end{align}
Let $\hat{g}_t$ be a Ricci-Deturck flow solution with background metrics $\bar{g}=g_t$. 
That is, $g_t$ and $\hat{g}_t$ satisfy the following evolution equations:
\begin{align}
\partial_t g_t &= P_{\bar{g}}(g_t)=-2Rc(g_t), \label{eqn:SI01_8}\\
\partial_t \hat{g}_t &= P_{\bar{g}}(\hat{g}_t)
=-2Rc(\hat{g}_t) + \mathcal{L}_{X_{g(t)}(\hat{g}_t)} \hat{g}_t.   \label{eqn:SI01_9}
\end{align}
Define
\begin{align}
h := h_t = \hat{g}_t - g_t.   \label{eqn:SI01_10}
\end{align}
Applying Proposition~\ref{prop:RD_operator_compute}, we note that $\bar{\gamma}_{ij}^k=0$ and obtain 
\begin{align*}
\partial_t h 
=P_{\bar{g}}(g_t + h) - P_{\bar{g}}(g_t) = \Delta_L h + \cQ [h]
=\Delta h + E \ast h + \nabla^\ast \cS [h] + \cR [h]
\end{align*}
where
\begin{align*}
    (E*h)_{ij}=2R_{iklj}h_{kl}-R_{ik}h_{kj}-R_{jk}h_{ki}. 
\end{align*}
All curvatures are calculated with respect to $g_t$ and therefore depend on time $t$. 
In short, we have
\begin{align}
    (\partial_t-\Delta)h=E*h +\cQ [h]. \label{equ:ptb_RD}
\end{align}
As the curvature is bounded along the Ricci flow by (\ref{eqn:SI01_6}), we have
\begin{align}
            |E| \leq C, \quad
            |\cS [h]| \leq  C|h||\nabla h|, \quad
            |\cR [h]| \leq  C(|h|^2 + |\nabla h|^2),
        \label{eqn:SI01_7}   
 \end{align}   
 where $|h|<\frac12$ and $C=C(n)$. \\
	
Note that the operators $\Rm,\nabla, \nabla^*, \Delta$ and $L=\Delta_L$ are all calculated with respect to $g_t$ by default.  By the assumption of curvature bound, we know that $g_t$ is equivalent to $g_0$ uniformly.
Let $K$ be the fundamental solution for the operator $\partial_t - L$ in \eqref{equ:ptb_RD}. 
The perturbation equation (\ref{equ:ptb_RD}) is equivalent to the following integral equation.
\begin{align}
	\begin{aligned}
		h(x,t) &= \int_{M} K(x,t;y,0) h_0(y) \dd g_0(y)\\
		  &\quad +\int_{M\times [0,t]} \big( K(x,t;y,s)\cR [h](y,s) 
           +\nabla_y K(x,t;y,s) \cS [h](y,s) \big)  \dd g_s(y) \dd s.
	\end{aligned}
\label{equ:int_ptb}    
\end{align}
Similarly to the previous section, we shall obtain a solution of~\eqref{equ:int_ptb} by the contraction mapping principle. 
For this purpose, we need to define different norms similar to (\ref{1def:local_X})-(\ref{eqn:X_and_Y_norm}).  
The only difference is that the metric is now evolving. We choose metric $g_0$ as the default metric to define balls: $B(x,r)=B_{g_0}(x,r)$.  However,  for the application of integration by parts, the default volume element we use is $\dd g_t$, which is evolving over time in a uniformly bounded way:
\begin{align*}
 \frac{1}{C} \dd g_0 
 \leq \dd g_t \leq C \dd g_0. 
\end{align*}
Instead of copying all the definitions here, we shall only take examples to show the difference. For example, we define
\begin{align*}
    &P(x,r^2):=B_{g_0}(x,r) \times (0, r^2), \\
    &\Omega(x, r^2):=B_{g_0}(x,r) \times (\frac{r^2}{2}, r^2). 
\end{align*}
For a smooth function $f$ defined on $M \times [0, T]$, we have
\begin{align*}
    &\|\nabla f\|_{L^2(P(x,r^2))}^2:=\int_{P(x,r^2)} |\nabla f|_{g_t}^2 dg_t dt, \\
    &\| \nabla f \|_{L^2(P(x,r^2))}^{n+4} 
      :=\int_{\Omega(x,r^2)} |\nabla f|_{g_t}^{n+4} dg_t dt, \\
    &\|f\|_{X_{B(x,r)}} :=\| f \|_{L^\infty(P(x,r^2))}
         + \frac{\| \nabla f \|_{L^2(P(x,r^2))}}{|B_{g_0}(x,r)|_{dg_0}^{\frac{1}{2}}}
         +r^{\frac{n+2}{n+4}} 
         \frac{\| \nabla f \|_{L^{n+4}(\Omega(x,r^2))}}{|B_{g_0}(x,r)|_{dg_0}^{\frac{1}{n+4}}}. 
\end{align*}

Similarly to Proposition~\ref{1clm:X<Y} and~\ref{1clm:Y<X^2}, we have the following two estimates: Proposition~\ref{clm:X<Y} and~\ref{clm:Y<X^2}.

\begin{prop}\label{clm:X<Y}
	Suppose $0<T \leq 1$. We have
    \begin{align}
       \left \| \int_{M\times [0,t]} \big( K(\cdot ,\cdot ;y,s) \cR(y,s) + \nabla_y K(\cdot ,\cdot ;y,s) \cS(y,s) \big)  \dd g_s(y) \dd s \right \|_{X_T} \leq C\|\cQ\|_{Y_T},  
    \label{eqn:SI01_14}   
    \end{align}
	where $C=C(n)$. 
\end{prop}

\begin{proof}
Following the route of the proof of Proposition~\ref{1clm:X<Y}, we set 
\begin{align}
		\xi(x,t):=  \int_{M\times [0,t]} \big( K(x,t;y,s)\cR(y,s) + \nabla_y K(x,t;y,s) \cS(y,s) \big)  \dd g_s(y) \dd s. 
		\label{eqn:h0add}    
\end{align}
Note that $\nabla_y$ is $\nabla_y^{g_s}$, the gradient calculated with respect to $g_s$. 
As the background metric $g_t$ is evolving, when $s\neq t$, it is possible that $\nabla_y^{g_s} \neq \nabla_y^{g_t}$. This is an important difference compared to the static case. 
The tensor field $\xi$ satisfies the equation
\begin{align*}
    (\partial_t -L) \xi = \cR + \nabla^* \cS. 
\end{align*}
By virtue of (\ref{eqn:h0add}),  in order to prove (\ref{eqn:SI01_14}), it suffices to show 
\begin{align*}
    \| \xi \|_{X_{B(x_0,r)}} \leq C\|\cQ\|_{Y_T}, \quad \forall x_0 \in M, \; r \in (0, \sqrt{T}), 
\end{align*}
where
 \begin{align}
        \| \xi \|_{X_{B(x_0,r)}} =\| \xi \|_{L^\infty(P(x_0,r^2))}
         + \frac{\| \nabla \xi \|_{L^2(P(x_0,r^2))}}{|B(x_0,r)|^{\frac{1}{2}}}
         +r^{\frac{n+2}{n+4}} 
         \frac{\| \nabla \xi \|_{L^{n+4}(\Omega(x_0,r^2))}}{|B(x_0,r)|^{\frac{1}{n+4}}}. 
    \label{eqn:SI02_1}     
 \end{align}
 The same as before, we shall estimate (\ref{eqn:SI02_1}) term by term. \\

	\textit{Step 1. Estimate of $\|\xi\|_{L^\infty(P(x_0,r^2))}$.} \\
	
	Take any $(x,t) \in B(x_0, r)\times [0,r^2] $ and denote the domain 
    $\Omega (x,t) = B_{g_0}(x, \sqrt t ) \times (t/2, t) $.  
    It follows from the definition (\ref{eqn:h0add}) that 
	\begin{align*}
			|\xi(x,t)|
		 &\leq   \underbrace{\int_{ \Omega(x,t)} \big(| K(x,t;y,s)\cR(y,s)| +| \nabla_y  K(x,t;y,s) \cS(y,s) | \big)  \dd g_s(y) \dd s}_{I} \notag\\
		 &\quad + \underbrace{\int_{M\times [0,t] \backslash \Omega(x,t)}  
         \big(| K(x,t;y,s)\cR(y,s)| +|\nabla_y  K(x,t;y,s) \cS(y,s)| \big)  \dd g_s(y) \dd s}_{II}. 
	\end{align*}
	The estimate of I is the same as the corresponding part in the proof of Proposition \ref{1clm:X<Y}.
    The estimate of II is almost the same. The only difference is that the metric and volume element is evolving now. However, they are all equivalent.  That is, for every $t \in [0, T]$, we have 
    \begin{align*}     
     \frac{1}{C} g_0(y) \leq g_t (y) \leq C g_0 (y), \quad   \frac{1}{C} \dd g_0(y) \leq \dd g_t(y) \leq C \dd g_0(y). 
    \end{align*}
    Thus, we can also use $g_0$ and $dg_0$ as the default metric and volume element. 
    Note that the heat kernel estimate and its gradient estimate in the current setting (cf. Proposition~\ref{prop:SI06_flow}) are the same as in the static case.  Therefore, using the same argument as in the proof of (\ref{1estm:II_bar}) and (\ref{eqn:SA12_4}), we have
    \begin{align*}
        II \leq C \sum_i e^{-\frac{d^2(x, p_i)}{4Dt}} \|\cQ \|_{Y_T} \leq C \|\cQ \|_{Y_T}. 
    \end{align*}
	Combining the estimates of I and II, we obtain 
    \begin{align}
        |\xi(x,t )| \leq C  \|  \cQ  \|_{Y_T}.   \label{eqn:SI01_19}
    \end{align}
	
	\textit{Step 2. Estimate of $\| \nabla \xi \|_{L^2(P(x_0,r^2))}$.}\\
	
	We choose a cut-off function $\eta $ such that $\eta = 1 $ in $B_{g_0}(x_0,r) $ and $\eta =0 $ out of $B_{g_0}(x_0,2r) $ with $|\nabla \eta|_{g_0} \leq C/r$.
    Since all metrics $g_t$ are equivalent, it is clear that 
    \begin{align*}
         |\nabla \eta|= |\nabla \eta|_{g_t} \leq \frac{C}{r}
    \end{align*}
    holds forever. 
    By (\ref{eqn:h0add}), it is clear that $\xi$ solves the differential equation 
    $\partial_t \xi- L \xi = \cR + \nabla^* \cS$ in the weak sense with the initial value $\xi_0=0$. 
    Therefore, for some constant $A>0$ to be determined, we have
    \begin{align*}
       0 &\leq \frac{1}{2} \int_M \eta ^2 e^{-Ar^2}|\xi(\cdot, r^2)|_{g_{r^2}}^2 dg_{r^2}=\frac{1}{2} \int_M \eta ^2 e^{-Ar^2} |\xi(\cdot, r^2)|_{g_{r^2}}^2 dg_{r^2}
       -\frac{1}{2} \int_M \eta ^2 |\xi(\cdot, 0)|_{g_0}^2 dg_{0}\\
       &=\frac{1}{2} \int_0^{r^2} \partial_t \left\{\int_M \eta^2 e^{-At}|\xi|_{g_t}^2 dg_t \right\} dt\\
       &=\frac{1}{2} \int_0^{r^2} \int_M  \eta^2 e^{-At} 
       \left( 
       2\langle \dot{\xi}, \xi \rangle -A|\xi|^2+ 4 R^{ik}\xi_{ij}\xi_{kj} -R|\xi|^2 \right) dg_t dt,
    \end{align*}
    where $R$ is the scalar curvature. 
    By the assumption of curvature bound, we may choose $A=A_n=100n^2$ such that
    \begin{align*}
        -A|\xi|^2+ 4 R^{ik}\xi_{ij}\xi_{kj}- R |\xi|^2 \leq (-A+4|Rc|+|R|) |\xi|^2 \leq 0. 
    \end{align*}
    It follows that
    \begin{align}
      \int_0^{r^2} \int_M  \eta^2 e^{A(r^2-t)} 
       \langle -\dot{\xi}, \xi \rangle  dg_t dt  
       \leq \frac{1}{2} \int_M \eta^2 |\xi(\cdot, r^2)|_{g_{r^2}}^2 dg_{r^2}. 
    \label{eqn:SI01_11}   
    \end{align}
    Recall that
    \begin{align*}
         \dot{\xi}=\Delta_L \xi +\cQ [h]=\Delta \xi + E*\xi + \cR [h] + \nabla^* \cS [h].
    \end{align*}
    Plugging it into (\ref{eqn:SI01_11}) yields 
    \begin{align*}
      &\quad \int_0^{r^2} \int_M  \eta^2 e^{A(r^2-t)} 
       \langle -\Delta \xi, \xi \rangle_{g_t}  dg_t dt \\
      &\leq \frac{1}{2} \int_M \eta^2 |\xi(\cdot, r^2)|_{g_{r^2}}^2 dg_{r^2}
         + \int_0^{r^2} \int_M  \eta^2 e^{A(r^2-t)} \left( E*\xi + \cR [h] + \nabla^* \cS [h] \right) dg_t dt.  
    \end{align*}
    Omitting the obvious volume elements and doing integration by parts imply 
     \begin{align*}
      &\quad \int_0^{r^2} \int_M   e^{A(r^2-t)}  \eta^2 |\nabla \xi|^2     \\
      &\leq \frac{1}{2} \int_M \eta^2 |\xi(\cdot, r^2)|_{g_{r^2}}^2 dg_{r^2}
         + \int_0^{r^2} \int_M   e^{A(r^2-t)} \left\{ \eta^2 (|E||\xi| + |\cR|) + 2\eta|\nabla \eta| (|\cS|+  |\nabla \xi|) \right\}. 
    \end{align*}
    Since $|E|\leq C$ and $|\nabla \eta| \leq \frac{C}{r}$, applying the elementary Cauchy-Schwarz inequality on the right hand side of the above inequality yields 
    \begin{align}
		\int_0^{r^2} \int_{B(x_0,r)}|  \nabla \xi |^2 
        &\leq C \int_0^{r^2} \int_{B(x_0,2r)} 
        \left\{ (r^{-2}+1) |\xi|^2+|\cR| |\xi|+ |\cS|^2 \right\},  \label{eqn:SI01_13}
	\end{align}
    which is almost the same as (\ref{eqn:SI01_12}), except that the integration is done with evolving metrics $g_t$. 

    Recall that the metrics $g_t$ are uniformly equivalent to $g_0$ by the assumption of curvature bound (\ref{eqn:SI01_6}). Following the same volume comparison and covering argument as in the proof of (\ref{eqn:SC01_12}) and (\ref{eqn:SC01_7}), we can use (\ref{eqn:SI01_19}) to obtain 
     \begin{align}
       \frac{\| \nabla \xi \|_{L^2(P(x_0,r^2))}}{|B(x_0,r)|^{\frac{1}{2}}} 
       \leq C  \|\cQ \|_{Y_T}.   \label{eqn:SI01_16}
    \end{align}

	\textit{Step 3. Estimate of $\|\nabla \xi\|_{L^{n+4}(\Omega(x_0,r^2))}$.}\\
	
	Suppose $(x,t)\in \Omega(x_0,r^2)   $ as in \textit{Step 1} and we split the integral into two parts.
	\begin{align*}
			|\nabla \xi |(x,t) = {}& \int_{M \times [0,t]} |\nabla_x K(x,t;y,s)\cR(y,s) + \nabla_x \nabla_y K(x,t;y,s)\cS(y,s)| \dd g_s(y) \dd s\\
			\leq {}&  \underbrace{\int_{ M\times[0,t] \cap \Omega(x_0,r^2)}   |\nabla_x K(x,t;y,s)\cR(y,s)  + \nabla_x \nabla_y K(x,t;y,s)\cS(y,s)| \dd g_s(y) \dd s}_{Inn(x,t)}\\
			{}& + \underbrace{\int_{M \backslash \Omega(x_0,r^2) } |\nabla_x K(x,t;y,s)\cR(y,s)  + \nabla_x \nabla_y K(x,t;y,s)\cS(y,s)| \dd g_s(y) \dd s}_{Out(x,t)}.
	\end{align*} 
	One can argue as in Step 1 above and obtain that 
    $$Out (x,t) < C t ^{-\frac{1}{2}} \|  \cQ  \|_{Y_T}. $$ 
    Thus,
	\begin{align}\label{estm:Out}
			r^{\frac{n+2}{n+4}}|B(x_0,r) |^{-\frac{1}{n+4}} \|  Out  \|_{L^{n+4}(\Omega(x_0,r^2))} \leq C   \|  \cQ   \|_{Y_T} .
	\end{align} 
	
	So we only need to estimate the $L^{n+4}(\Omega(x_0,r^2))$ norm 
    of $Inn$. Without loss of generality, we assume that $\cR, \cS $ are supported in $\Omega(x_0, r^2) $. For any $(x,t) \in \Omega (x_0,r^2)$ we define $I_R(x,t) $ and $I_S(x,t) $ such that $Inn (x,t) = I_R(x,t) +  I_S (x,t) $ as follows. 
    \begin{align*}
      &I_R (x,t):=\int_{\frac{r^2}{2}} ^{t} \int_{M} \nabla_x  K (x,t;y,s)\cR(y,s)\dd g_s (y) \dd s,\\
      &I_S (x,t):=\int_{\frac{r^2}{2}}  ^{t} \int_{M} \nabla_x \nabla_y K (x,t;y,s)\cS(y,s)\dd g_s (y) \dd s.
    \end{align*}
	The term $I_R$ can be estimated the same as~\eqref{eqn:SA09_5} and~\eqref{1estm:I_R bar}. We have
    \begin{align}
        r^{\frac{n+2}{n+4}}|B(x_0,r) |^{-\frac{1}{n+4}}	\| I_R \|_{L^{n+4}(\Omega(x_0,r^2))} 
        \leq C  \|\cR\|_{Y_{B(x_0,r)}^0}. 
        \label{eqn:SI01_20}
	\end{align}
    The estimate of $I_S$ is more delicate. We define  
    \begin{align*}
        \xi_S(x,t)
        := \int_{\frac{r^2}{2}} ^{t} \int_{M} \nabla_y K (x,t;y,s)\cS(y,s)\dd g_s (y) \dd s.
    \end{align*}
    Then we have
    \begin{align}
        \nabla_x  \xi_S(x,t) = I_S(x,t). \label{eqn:SI01_17}
    \end{align}
	Note that $\xi_S$ satisfies the equation
    \begin{align}
        (\partial_t-L) \xi_S= \nabla^\ast \cS.   \label{eqn:SI04_8} 
    \end{align}
    with initial value $\xi_S|_{t=r^2/2} =0$.

    Since $\xi_S$ satisfies (\ref{eqn:SI04_8}), we can apply the same method to deduce (\ref{eqn:SI01_13}) to obtain 
    \begin{align}
		\int_0^{r^2} \int_{B(x_0,r)}|  \nabla \xi_S |^2 
        \leq C \int_0^{r^2} \int_{B(x_0,2r)} 
        \left\{ (r^{-2}+1) |\xi_S|^2+ |\cS|^2 \right\}. 
    \label{eqn:SI04_10}    
	\end{align}
    The same as (\ref{eqn:SI04_9}), we have the following.
      \begin{align}
     \|\xi_S\|_{L^2(\Omega(x_0, r^2))}
     \leq C r \|\cS\|_{L^2(\Omega(x_0, r^2))}.  
    \end{align}
    Consequently, 
	\begin{align*}
		 \int_{\frac{r^2}{2}}^{r^2}\int_{B_{g_0}(x_0,r)} |\nabla \xi_S|^2 dg_t dt
		\leq C  \|  \cS  \| ^2_{L^2(\Omega(x_0, r^2))}. 
	\end{align*}
	By (\ref{eqn:SI01_17}), the above inequality can be understood as 
	\begin{align}
			{}& \left  \|  \int_0^t\int_{M} \nabla_x^{g_t} \nabla_y^{g_s} K (x, t;y,s)\cS(y,s)\dd g_s(y) \dd s \right \|_{L^2(\Omega(x_0, r^2))}
            \leq C  \|  \cS  \|_{L^2(\Omega(x_0, r^2))}.
    \label{estm:I_S L^2}
	\end{align}
	This means that the integral operator
    $\mathcal{K}=\nabla_x^{g_t} \nabla_y^{g_s} K$ is bounded in space $L^2(\Omega(x_0,r^2))$. 
 Using the Calder\'on-Zygmund inequality of Han-Wang (cf.~\cite{hanwang_CZ}) with $p=n+4>2$, we obtain
	\begin{align*}
			 \|I_S\|_{L^{n+4}(\Omega(x_0, r^2))}
			\leq C  \|\cS \|_{L^{n+4}(\Omega(x_0, r^2))}, 
	\end{align*}
	which yields
    \begin{align}
       r^{\frac{n+2}{n+4}}|B(x_0,r) |^{-\frac{1}{n+4}} \|  I_S  \|_{L^{n+4}(\Omega(x_0, r^2))}
	\leq  C  \|  \cS  \|_{Y^1(x,r)}. 
    \label{eqn:SI01_18} 
    \end{align}
	Combining estimates of $Out$ , $I_R$ and $I_S $ in \eqref{estm:Out}, \eqref{eqn:SI01_20} and \eqref{eqn:SI01_18}, we obtain
	$$r^{\frac{n+2}{n+4}}|B(x_0,r) |^{-\frac{1}{n+4}} \|  \nabla \xi  \|_{L^{n+4}(\Omega(x_0,r^2))} \leq C  \|  \cQ   \|_{Y_T}. $$ 
\end{proof}

\begin{prop}
\label{clm:Y<X^2}
    There exists a constant $\delta<1 $ such that if $ \|h\|_{X_T} <\delta $, then for each $x \in M$
    and $r \in (0, \sqrt{T}]$, we have
	\begin{align}
    \label{Y<X}
			 \|\cQ [h]\|_{Y_{B(x,r)}} \leq  C(n, T) \|h\|_{X_{B(x,r)}}^2. 
	\end{align}
	Moreover, if $ \max \{\|v\|_{X_T}, \|w\|_{X_T} \} < \delta$, then for each $x \in M$
    and $r \in (0, \sqrt{T}]$, we have
	\begin{align}
    \label{Y<X_difference}
			 \|\cQ [v] - \cQ [w]\|_{Y_{B(x,r)}} 
             \leq C(n,T) (\|v\|_{X_{B(x,r)}} + \|w\|_{X_{B(x,r)}}) \|v-w\|_{X_{B(x,r)}}. 
	\end{align}
\end{prop}

The proof of this proposition is the same as that of Proposition~\ref{1clm:Y<X^2}. So we omit it. 

\begin{prop}
\label{clm:X<infty}
	We have
    \begin{align}
      \left\|  \int_{M} K(x,t;y,0) h_0 (y) \dd g_0 (y) \right\|_{X_T} 
      \leq C \|h_0\|_{L^\infty (M)}, 
    \label{eqn:SI04_12}  
    \end{align}
	where $C=C(n,T)$.
\end{prop}

\begin{proof}
   This is almost the same as the proof of Proposition~\ref{prop:K_conv_Z}.
\end{proof}

Then we can apply the contraction mapping principle to solve the integral equation~\eqref{equ:int_ptb}. 

\begin{thm}
\label{thm:SJ15_11}
For any metric $\hat{g}_0$ on $M$ with $\|\hat{g}_0 - g_0\|_{L^{\infty}(M)} < \epsilon $, there exists a smooth 
solution 
$(h_t)_{t\in [0,T]} \in X_T$ of~\eqref{equ:int_ptb}, with the initial condition $h_0=\hat{g}_0-g_0$. 
Moreover, the following estimates hold. 
    \begin{align}
      \|h\|_{L^\infty(M \times [0, T])} \leq \|h\|_{X_T} 
      \leq C  \|h_0\|_{L^{\infty}(M)}.
    \label{eqn:SI04_11}  
    \end{align}
\end{thm}

\begin{proof}
    Define $\phi_0 :=\hat{g}_0-g_0$. Suppose $\|\phi_0\|_{L^{\infty}(M)}<\delta$ where $\delta$ is a small positive constant to be determined later. 
    
      Define $\Phi(\cdot )$ as the following operator. 
	\begin{align}
    \begin{aligned}
			\Phi(h)  
            &:=\int_{M} K(x,t;y,0) \phi_0(y) \dd g_0(y) \\
            &\quad +\int_{M\times [0,t]} \big( K(x,t;y,s)\cR [h](y,s) + \nabla_y K(x,t;y,s) \cS [h](y,s) \big)  \dd g_s(y) \dd s.
    \end{aligned}
    \label{def:int_oprt}
	\end{align} 
    If $\|h\|_{X_T}<\delta$, it follows from 
    Proposition~\ref{clm:X<Y} and Proposition \ref{clm:Y<X^2} that
    \begin{align*}
      \|\Phi(h)\|_{X_T} < C \|h\|_{X_T}^2<C\delta^2< \delta. 
    \end{align*}
    Therefore, $\Phi$ maps $B_{\delta,T}:=\{h\in X_T: \|h\|_{X_T} \leq \delta\}$ to itself. 

    Now we choose two elements $v,w \in B_{\delta,T}$. 
	Applying Proposition~\ref{clm:X<Y} and Proposition~\ref{clm:Y<X^2} again, we have
	\begin{align*}
			&\quad \left  \|  \Phi_{\phi_0 }(v ) - \Phi_{\phi_0}(w) \right  \|_{X_T} \\
			&=\bigg\|   \int_{M\times [0,t]} \bigg( K(x,t;y,s)(\cR [v] - \cR [w])(y,s) 
			  +\nabla_y K(x,t;y,s) (\cS [v] - \cS [w])(y,s) \bigg)  \dd g_s(y) \dd s \bigg\|_{X_T}\\
            &\leq   C  \|  \cQ [v] - \cQ [w]  \|_{Y_T } \leq C (\|v\|_{X_T} + \|w\|_{X_T}) \|v-w\|_{X_T}.
	\end{align*}
	Now we take $\delta$ small enough such that $ 2 C\delta <\frac12$.
    It follows from the above inequality that
	\begin{align*}
			 \|  \Phi_{\phi_0 }(v ) - \Phi_{\phi_0}(w)  \|_{X_T} 
			\leq 2C \delta \|v-w\|_{X_T}
            <\frac12  \|v-w\|_{X_T}.
	\end{align*}
    Therefore, $\Phi_{\phi_0} $ is a contraction mapping that maps from $B_{\delta, T}$ to $B_{\delta,T}$.
    According to Banach's fixed point principle, there exists a unique $h \in B_{\delta, T}$ that
 satisfies $h=\Phi(h)$.  In other words, $h$ satisfies~\eqref{equ:int_ptb}. 
    
    Since the heat kernel $K$ is smooth, it is clear from convolution that $h$ is smooth and satisfies the initial condition $h_0=\phi_0$. 
    Using Proposition~\ref{clm:X<Y}, Proposition~\ref{clm:Y<X^2} 
    and Proposition~\ref{clm:X<infty} again, we know $h$ satisfies
	\begin{align*}
			\|h\|_{X_T}
			\leq 
            \left\|  \int_{M} K(x,t;y,0) h_0 (y) \dd g_0 (y) \right\|_{X_T} 
             +C\|h\|_{X_T}^2, 
	\end{align*}
    which implies that
    \begin{align*}
       \frac12 \|h\|_{X_T}
       \leq \left(1-C\|h\|_{X_T}\right) \|h\|_{X_T}
       \leq \left\|  \int_{M} K(x,t;y,0) h_0 (y) \dd g_0 (y) \right\|_{X_T} 
       \leq C  \|  h_0  \|_{L^\infty}, 
    \end{align*}
    where we applied (\ref{eqn:SI04_12}) in the last inequality.  By the definition of the $X_T$-norm, it is clear that (\ref{eqn:SI04_11}) follows directly from the above inequality. 
\end{proof}

\section{Proof of the main theorem}
\label{sec:final}

If $g_0$ is regular enough, say $\sum_{k=0}^3 |\nabla^k Rm| \leq C$, then the main theorem already follows from the combination of results in Section~\ref{sec:exist} and Section~\ref{sec:cd}.
Therefore, we only need to drop the assumptions $\sum_{k=0}^3 |\nabla^k Rm| \leq C$. This can be done in many ways.  We provide more details as follows. 

\begin{proof}[Proof of Theorem~\ref{thm:main}]
 It suffices to show the short-time existence of the Ricci flow solution initiated from a metric $g_0$ with $|Rm|_{g_0} \leq 1$.   For each such metric, we can find smooth metrics $g_i$ approximating $g_0$ without using Ricci flow.
 In fact, by embedding method, we can make $g_i$ satisfy (cf.~\cite{abresch1988glatten} or~\cite{petersen1999controlled}): 
    \begin{itemize}
        \item $e^{-\epsilon_i} g_0 \leq g_i \leq e^{\epsilon_i} g_0$, with $\epsilon_i \to 0$. 
        \item $|Rm|_{g_i} \leq \Lambda$. 
        \item $\sum_{k=0}^{3} |\nabla^k Rm|_{g_i} \leq C(n, \Lambda, \epsilon_i)$. 
    \end{itemize}
    Applying Theorem~\ref{1thm:int_pert_solu_exist} and the standard diffeomorphism action, we obtain the Ricci flow solutions 
    $\{g_{i,t}\}_{0 \leq t \leq T_i}$.
    By further extending the Ricci flow if possible, we can assume $T_i$ is the maximum existence time.   
    The maximum principle for Riemannian curvature evolution implies that $T_i \geq \frac{C_n}{\Lambda}$, which is uniformly away from $0$. Let $T:=\frac{C_n}{2\Lambda}$. 
    Then we obtain a family of Ricci flow solutions $\{g_{i,t}\}_{0 \leq t \leq T}$ on $M$ with bounded curvature $|Rm|_{g_{i,t}} \leq C \Lambda$.
    Using Shi-type estimate, we have higher order curvature derivatives bound
    $t^{\frac{k}{2}}|\nabla^k Rm| \leq C(k,n, \Lambda)$. Taking smooth limit of $g_{i,t}$, we obtain a Ricci flow solution $\{g_t\}_{0<t\leq T}$. 
    Also, the curvature estimate and evolution of metrics implies $\lim_{t \to 0} g_t=g_0$. 
    Therefore, we obtain a Ricci flow solution $\{g_t\}_{0 \leq t \leq T}$, which initiates from $g_0$ and has bounded sectional curvature. 

    There are other ways to obtain the short-time existence directly. 
    For example, by further work on the heat kernel estimates and Caldr\'on-Zygmund inequality (cf.~\cite{hanwang_CZ}), one can obtain the result of Theorem~\ref{1thm:int_pert_solu_exist} directly only 
    under the curvature $|Rm|_{g_0} \leq 1$. 
    
    The short-time existence is proved.
    The uniqueness and continuous dependence follow from Theorem~\ref{thm:SJ15_11}. 

\end{proof}
  
\bibliography{refmain}

\vskip10pt

Jingbin Cai, School of Mathematical Sciences, University of
Science and Technology of China, No. 96 Jinzhai Road, Hefei, Anhui Province, 230026, China;
binge@mail.ustc.edu.cn.\\

Bing Wang, Institute of Geometry and Physics, School of Mathematical Sciences, University of Science and Technology of China, No. 96 Jinzhai Road, Hefei, Anhui Province, 230026, China; Hefei National Laboratory, Hefei, 230088, China. topspin@ustc.edu.cn.\\

\end{document}